\documentclass[10pt, a4paper, reqno, english]{amsart}

\usepackage{hyperref}
\usepackage{url}
\usepackage{amssymb,mathrsfs,mathtools}
\usepackage[all]{xy}
\usepackage{enumitem}
\usepackage{tikz}
\usetikzlibrary{matrix,arrows}

\newcommand\ZZ{\mathbb{Z}}

\newcommand\NN{\mathbb{N}}

\newcommand\RR{\mathbb{R}}
\newcommand\QQ{\mathbb{Q}}

\newcommand\PP{\mathbb{P}} % projective space
\newcommand{\A}{\mathbb{A}} %% affine space
\newcommand\GG{\mathbb{G}} % algebraic groups

\newcommand{\OO}{\mathcal{O}} % structure sheaf

\newcommand{\et}{{\text{\it\'et}}}

\DeclareMathOperator{\sing}{Sing} % singular locus
\DeclareMathOperator{\spec}{Spec} % spectrum of ring
\DeclareMathOperator{\nr}{nr} % unramified
\DeclareMathOperator{\pic}{Pic} % picard group
\DeclareMathOperator{\Aut}{\bf{Aut}} % automorphism functor
\DeclareMathOperator{\aut}{Aut} % automorphism group
\DeclareMathOperator{\gal}{Gal} % galois group
\DeclareMathOperator{\Br}{Br} % brauer group
\DeclareMathOperator{\PGL}{PGL} % projective linear group
\DeclareMathOperator{\GL}{GL} % linear group
\DeclareMathOperator{\Gr}{Gr} % Grassmannian
\DeclareMathOperator{\ortgrass}{OGr} % orthogonal grassmannian
\DeclareMathOperator{\ortgrassfield}{Y} % form of orthogonal grassmannian
\DeclareMathOperator{\rk}{rk} % rank

\newtheorem{theorem}{Theorem}
\newtheorem{lemma}[theorem]{Lemma}
\newtheorem{prop}[theorem]{Proposition}
\newtheorem{cor}[theorem]{Corollary}
\theoremstyle{definition}
\newtheorem{defin}[theorem]{Definition}

\newtheorem{remark}[theorem]{Remark}
\newtheorem{notation}[theorem]{Notation}

\newtheorem{conjecture}[theorem]{Conjecture}

\numberwithin{theorem}{section}
\numberwithin{equation}{section}

\begin{document}

\setcounter{tocdepth}{1}

\title[On rationally connected varieties over $C_1$ fields of char. 0]{On rationally connected varieties over $C_1$ fields of characteristic 0}

\address{Mathematical Institute,
Utrecht University,
Budapestlaan 6,
3584 CD Utrecht, The Netherlands; EPFL SB MATH CAG, B\^at. MA, Station 8, 1015 Lausanne, Switzerland}
  
\author{Marta Pieropan} \email{m.pieropan@uu.nl}

\date{\today}

\begin{abstract}
We use birational geometry to show that the existence of rational points on proper rationally connected varieties over fields of characteristic 0 is a consequence of the existence of rational points on terminal Fano varieties. We discuss several consequences of this result, especially in relation to the $C_1$-conjecture. We also provide evidence that supports the conjecture in dimension 3 for $C_1$ fields of characteristic 0.
\end{abstract}

\maketitle

\tableofcontents

\section{Introduction}
A field $k$ is called $C_1$ (or quasi algebraically closed) 
if every hypersurface of degree at most $n$ in $\PP^n_k$ has a $k$-rational point.
Quasi algebraically closed fields were introduced by Artin and first extensively studied by Tsen (see \cite{MR1743370}) 
and Lang \cite{MR0046388}.
Smooth hypersurfaces of degree at most $n$ in $\PP^n$ are Fano and rationally chain connected \cite{MR1189503, MR1191735}. In characteristic 0 they are also rationally connected. Hence, it is natural to study rational points on rationally (chain) connected varieties over $C_1$ fields. 
In an unpublished paper of 2000 Lang formulated the following conjecture, also known as $C_1$-conjecture.
\begin{conjecture}[Lang, 2000]\label{conj:C_1}  
Every smooth proper separably rationally connected variety over a $C_1$ field has a rational point.
\end{conjecture}
At the time it was known that smooth proper rational curves and surfaces have points over $C_1$ fields \cite{MR934267}. 
Soon after it was formulated, the conjecture has been proven to hold for the following $C_1$ fields: finite fields \cite{MR1943746}, function fields of curves over algebraically closed fields \cite{MR1937041, MR1937199, MR1981034}, fields of formal power series over algebraically closed fields \cite{MR2757627}. 
It is still open for the maximal unramified extensions $\QQ_p^{\nr}$ of $p$-adic fields.
Positive evidence in support of the conjecture for the latter class of fields is given by \cite{MR3760308}, which proves that for sufficiently large primes $p$  smooth projective rationally connected varieties over $\QQ_p^{\nr}$ with fixed Hilbert polynomial have a rational point. 
As a consequence of our first main result  we manage to improve \cite{MR3760308}: we replace the dependence on the Hilbert polynomial by dependence only on the dimension, and we remove the smoothness and properness assumptions.

The goal of this paper is to use birational geometry in characteristic 0 to reduce the $C_1$-conjecture to the case of Fano varieties and to provide evidence for the existence of rational points on rationally connected varieties of dimension $3$. Along the way, we also show that
for proper rationally connected varieties of fixed dimension over an arbitrary field of characteristic 0 there is a uniform upper bound for the  degree of the minimal field extensions where the set of rational points becomes nonempty. 

\subsection{Birational geometry}
The first main result of this paper reduces the $C_1$-conjecture for  $C_1$ fields of characteristic 0 to the following conjecture.
\begin{conjecture}\label{conj:terminal_Fano_final}
Every terminal 
$\QQ$-factorial 
Fano variety of Picard rank 1 over a  $C_1$ field of characteristic 0 has a rational point. 
\end{conjecture}

More precisely,  birational geometry and induction on dimension are used to prove the following statement. 

\begin{theorem}\label{thm:equivalent_formulation}
Let $k$ be a field of characteristic 0. For every positive integer $n$ the following statements are equivalent:
\begin{enumerate}[label=(\roman*), ref=(\roman*)]
\item Every smooth proper rationally connected $k$-variety of dimension $\leq n$ has a $k$-point. \label{item:thm:equiv:ratconn}
\item Every terminal  $\QQ$-factorial  Fano variety of dimension $\leq n$ and Picard rank 1 over $k$ has a $k$-point. \label{item:thm:equiv:fano}
\end{enumerate}
\end{theorem}

The proof of Theorem \ref{thm:equivalent_formulation} rests upon the birational invariance of rational connectedness and the fact that the existence of rational points is a birational invariant among proper smooth varieties \cite{MR0095851}. We recall that the use of the Minimal Model Program in dimension $\geq3$ produces singular birational models in general. Results on degenerations of rationally connected varieties \cite{MR2491906} and induction on dimension are used to transfer the existence of rational points among birationally equivalent varieties with mild singularities. The restriction on the characteristic is due to the use of the Minimal Model Program, of resolution of singularities and of \cite{MR2491906}. None of these are known to hold in positive characteristic, except for some results in low dimension. 

Boundedness of terminal Fano varieties \cite{MR4224714} assures that there are only finitely many deformation families of terminal Fano varieties in each fixed dimension over algebraically closed fields of characteristic 0.
As an application of this fact we prove the following corollary of Theorem \ref{thm:equivalent_formulation}, which gives a generalization
of \cite{MR3760308}.
 
\begin{cor}\label{cor:duesler-knecht_enhanced}
For every positive integer $n$ there exists a finite set $S(n)$ of prime numbers such that for all prime numbers $p\notin S(n)$, every 
rationally connected variety of dimension $n$ over $\QQ_p^{\nr}$ has a $\QQ_p^{\nr}$-point.
\end{cor} 

During a workshop in Edinburgh in November 2018, it was brought to the author's attention that the strategy of the proof of Theorem \ref{thm:equivalent_formulation} is similar to the one used in the paper \cite{MR3483470}, which proves, among other things, a uniform bound depending only on dimension for the indices of subgroups of finite groups of  automorphisms such that the subgroup acts with fixed points. The 
corresponding statement for rational points is the following theorem, which is proven by combining the proof of Theorem \ref{thm:equivalent_formulation}  with boundedness of terminal Fano varieties. 
 
 \begin{theorem}\label{thm:extension_points}
 For every positive integer $n$ there exists a positive integer $d(n)$ such that for every field $k$ of characteristic 0 and every proper rationally connected $k$-variety $X$ of dimension $n$ there exists a field extension $k'$ of $k$ of degree $\leq d(n)$ such that $X(k')\neq\emptyset$. 
 \end{theorem}
  For every positive integer $n$, let $\tilde d(n)$ be the smallest positive integer such that  for every field $k$ of characteristic 0 and every proper rationally connected $k$-variety $X$ of dimension $n$ there exists a field extension $k'$ of $k$ of degree $\leq \tilde d(n)$ such that $X(k')\neq\emptyset$. We prove the following effective bounds.
 \begin{theorem}\label{thm:extension_points_effective}
 We have
 $\tilde d(1)=2$, $\tilde d(2)=6$, $\tilde d(3)\leq 144$.
 \end{theorem}
 The proof of Theorem \ref{thm:extension_points} gives a bound  $$\tilde d(n) \leq \tilde d(n-1)\max\{d'(n), \max_{1\leq m\leq n-1}\tilde d(m)\tilde d(n-m)\},$$ where $d'(n)$ is an upper bound for the degree of the minimal field extensions where terminal Fano varieties of dimension $n$ acquire rational points. The existence of $d'(n)$ is a consequence of boundedness of terminal Fano varieties. The effective computability of $d'(n)$ depends on the classification of the Fano varieties that appear in  Theorem \ref{thm:equivalent_formulation}\ref{item:thm:equiv:fano}. In dimension $4$ or higher, the classification is not complete. 
In dimension $2$ the inequality only gives an upper bound $\tilde d(2)\leq 12$. The sharp bound $\tilde d(2)=6$ is obtained in Section \ref{sec:dP_classification} using the Enriques--Iskovskih--Manin classification of surfaces over nonclosed fields \cite{MR525940}. In Section \ref{sec:fano_3folds_existence} we recall the classification of the terminal Fano threefolds that appear in Theorem \ref{thm:equivalent_formulation}\ref{item:thm:equiv:fano}, and we prove that $d'(3)\leq 24$. We do not expect the bound for $\tilde d(3)$ to be sharp.

\subsection{Classification of Fano varieties}\label{intro:sec:classification}
Further evidence for Conjecture \ref{conj:C_1} that can be found in the literature is given by the fact that the following rationally connected varieties have rational points over $C_1$ fields 
that admit normic forms of arbitrary degree:
 complete intersections in weighted projective spaces (see \cite{MR1440180} 
 for example) and  hypersurfaces of split toric varieties \cite{arXiv:1407.6945}. 
We refer to \cite{MR0046388} for the condition about normic forms and we recall that it is satisfied by fields that admit finite extensions of every degree.
 In Section \ref{sec:varieties_with_points} we prove similar results for varieties that after  extension of the base field to the algebraic closure belong to the following classes: toric varieties, complete intersections in products of weighted projective spaces, some Fano cyclic coverings, some special Fano varieties of dimension 3.

As a consequence of Theorem \ref{thm:equivalent_formulation} combined with boundedness of
terminal Fano varieties \cite{MR4224714}, 
 the verification of Conjecture \ref{conj:C_1} is reduced to finding rational points on finitely many geometric families of Fano varieties in each given dimension. 
 In general, the $\QQ$-factoriality and the Picard rank 1 conditions in Conjecture \ref{conj:terminal_Fano_final} and in Theorem \ref{thm:equivalent_formulation} are not preserved under base field extensions.
 So in order to verify Conjecture \ref{conj:terminal_Fano_final} one needs to find rational points also on some Fano varieties that are not geometrically $\QQ$-factorial and on some Fano varieties of geometric Picard rank $>1$.

We recall that the Gorenstein index of a log terminal Fano variety $X$ is the smallest positive integer $I(X)$ such that $I(X)K_X$ is Cartier. The $\QQ$-Fano index of $X$ is then the largest positive integer $q(X)$ such that $-I(X)K_X$ is linearly equivalent to $q(X)H$ for some Cartier divisor $H$ on $X$. We write $r(X)=q(X)/I(X)$. 
In general these indices do not need to be invariant under field extension. 
In Section \ref{sec:fano_3folds_existence} we show that  $r(X)=r(X_{\overline k})$ for $X$ defined over a $C_1$ field of characteristic 0 with algebraic closure $\overline k$.
In Section \ref{sec:fano_large_index}, we use Fujita's  and Sano's classification results \cite{MR664549, MR1399470} to find rational points on  Fano varieties of large index.
 
  \begin{theorem}\label{thm:largeindex_Fano}
 If $X$ is a terminal Fano variety of dimension $n\geq2$ and index $r(X)> n-2$ over a $C_1$ field $k$ of characteristic 0 that admits normic forms of arbitrary degree, then $X(k)\neq\emptyset$.
 \end{theorem}
 
The classification of terminal Fano varieties is far from known. However, the current literature on terminal Gorenstein Fano threefolds \cite{MR1944132, MR3100918} makes the following  result  possible.

\begin{theorem}\label{thm:3dim_Fano_gor}
Let $X$ be a  Gorenstein terminal $\QQ$-factorial Fano threefold  over a  $C_1$ field $k$ of characteristic 0 that admits normic forms of arbitrary degree. Assume that $\rk\pic(X)=1$. 
Let $g(X)=\frac 12(-K_X)^3+1$. Let $\overline k$ be an algebraic closure of $k$.  Then $X(k)\neq\emptyset$ except, possibly, in the following cases:
\begin{enumerate}[label=(\roman*), ref=(\roman*)]
\item $r(X)=1$, $\rk\pic(X_{\overline k})=1$  and $g(X)\in\{6,7,8,9,10\}$;
\label{item:thm:3dim_Fano_gor:grass}
\item $r(X)=1$, $\rk\pic(X_{\overline k})=1$  and $|-K_{X_{\overline k}}|$ has a movable decomposition.
\label{item:thm:3dim_Fano:decomposable}
\end{enumerate}
 \end{theorem}
 
The varieties in part \ref{item:thm:3dim_Fano_gor:grass} of the theorem are complete intersections of sections of certain vector bundles on some Grassmannians. To the best of the author's knowledge, Fano varieties  as in \ref{item:thm:3dim_Fano:decomposable} are not classified.  At the beginning of Section \ref{sec:sing_gor_rkpic_1} we recall the definition of movable decomposition  from \cite{MR1944132}. 
 
To date there is no complete classification of non-Gorenstein terminal  Fano threefolds of index $r\leq 1$.
There is, however, a classification of the possible configurations of non-Gorenstein singularities for geometrically $\QQ$-factorial non-Gorenstein terminal  Fano threefolds of geometric Picard rank 1.  See the Graded Ring Database \cite{GRD}. 
Such Fano varieties $X$ are studied using a Fano index  defined as the largest positive integer  $f(X)$ such that $-K_X$ is linearly equivalent to $f(X)A$ for some Weil divisor $A$ on $X$ \cite{MR2067795}. 
In Section \ref{subsection:non-gorenstein} we prove that if 
$f(X_{\overline k})\geq5$ 
the variety has a rational points over  $C_1$ fields of characteristic 0. We also discuss to what extent the same proof applies to varieties of lower Fano index.

To the author's knowledge there is no classification of non-Gorenstein terminal Fano threefolds of geometric Picard rank $>1$ except for \cite{MR1317287}. The varieties in \cite{MR1317287} have rational points over $C_1$ fields of characteristic 0 by Theorem \ref{thm:non_gor_cyclic}.

We recall that the condition about the existence of normic forms of arbitrary degree is satisfied by the fields $\QQ_p^{\nr}$, as they admit finite extensions of every degree. Their algebraic extensions, however, do not need to satisfy the condition on normic forms, but if Conjecture \ref{conj:C_1} holds for $\QQ_p^{\nr}$, then it holds also for all its algebraic extensions, as Weil restriction under finite separable field extensions preserves rational connectedness, smoothness and properness, and the Weil restriction of a variety $X$ from a finite extension $k$  to $\QQ_p^{\nr}$ has a $\QQ_p^{\nr}$-point if and only if $X$ has a $k$-point.

Theorem \ref{thm:equivalent_formulation} and Corollary \ref{cor:duesler-knecht_enhanced} are proven in Section \ref{section:reformulation},  Theorems \ref{thm:extension_points} and \ref{thm:extension_points_effective} are proven in Section \ref{sec:extension_points}, Theorem \ref{thm:largeindex_Fano} is proven in Section \ref{sec:fano_large_index}, Theorem \ref{thm:3dim_Fano_gor} is proven in  Section \ref{sec:fano_3fold}.

\subsection*{Acknowledgements}
This work was partially supported by grant ES 60/10-1 of the Deutsche Forschungsgemeinschaft.
The author is grateful to H.~Esnault for introducting her to the $C_1$-conjecture. The author wishes to thank  J.~Blanc, Y.~Gongyo, A.~H\"oring, Z.~Patakfalvi, Y.~Prokhorov,  R.~Svaldi for useful discussions, J.-L.~Colliot-Th\'el\`ene and the anonymous referee for their remarks that led to significant improvements of the paper, and K.~Shramov for informing her about the paper \cite{MR3483470} and for suggesting  the topic of Theorem \ref{thm:extension_points}. The author acquainted herself to the Minimal Model Program at the \emph{Introductory Workshop on MMP} held in Hannover in February 2016. 

\section{Notation and basic properties}
Let $k$ be a field of characteristic 0, and $\overline k$ an algebraic closure of $k$. We denote by $\Br(k)$ the Brauer group of  $k$. 

For us a $k$-variety is a separated scheme of finite type over $k$.
We use the words curve, surface, threefold to denote a variety of dimension 1, 2, 3, respectively.

We say that a $k$-variety $X$ is rationally connected if $X_{\overline k}$ is integral and rationally connected in the sense of \cite[\S IV.3]{MR1440180}. 
We recall that if $X$ is rationally connected, $X\dashrightarrow Y$ is a dominant rational map  and $Y$ is proper, then $Y$ is rationally connected. In particular, being rationally connected is a birational invariant of proper varieties. 

Weil divisors on an integral normal variety $X$ correspond to Cartier divisors on its smooth locus $U$. The reflexive sheaf $\OO_X(D)$ associated to a Weil divisor $D$ on $X$ is  the push-forward of $\OO_U(D|_U)$ under the inclusion $U\subseteq X$.
A canonical divisor $K_X$ on $X$ is the Weil divisor corresponding to a canonical divisor on $U$. 
We say that $X$ is Gorenstein if its canonical divisor $K_X$ is Cartier, $\QQ$-Gorenstein if a positive multiple of $K_X$ is Cartier.
We say that  $X$ is $\QQ$-factorial if every Weil divisor on $X$ has a positive multiple that is Cartier. We recall that $\QQ$-factoriality is not invariant under field extension (see \cite[Remark 2.7]{arXiv:1512.05003}).

We refer to \cite{MR3057950} for the singularities of the minimal model program. We recall that  (log) terminal varieties are normal and $\QQ$-Gorenstein by definition. We also recall that the notion of (log) terminal singularities is invariant under separable field extension. 

A (log) terminal variety $X$ is Fano if $-K_X$ has a positive multiple which is Cartier and ample.
We refer to \cite{MR1668579} for the theory of Fano varieties. 

We recall that, given a field $k$ and an algebraic closure $\overline k$ of $k$, the $k$-forms of a $k$-variety $Y$ are $k$-varieties $Y'$ such that $Y'_{\overline k}$ is isomoprhic to $Y_{\overline k}$, the $k$-models of a $\overline k$-variety $Z$ are $k$-varieties $Z'$ such that $Z_{\overline k}$ is isomorphic to $Z$.

\section{Reformulation of the $C_1$-conjecture}\label{section:reformulation}
In this section we prove that over a $C_1$ field of characteristic 0 the $C_1$-conjecture  is equivalent to Conjecture \ref{conj:terminal_Fano_final}. 
We start by observing that in characteristic 0 the $C_1$-conjecture  is equivalent to both the stronger version obtained by removing the smoothness assumption and to the weaker version obtained by replacing the properness assumption by projectivity, as the following lemma shows.
 
 \begin{lemma}\label{lem:conj_sing}
 Let $k$ be a field of characteristic 0 and $n$ a positive integer. If all smooth projective rationally connected $k$-varieties of dimension $n$ have  $k$-points, then every proper rationally connected $k$-variety of dimension $n$ has a $k$-point.
 \end{lemma}
 \begin{proof}
 Let $X$ be a proper rationally connected $k$-variety of dimension $n$. 
Since $X$ is of finite type over $k$, there exists an open subset $U$ of $X$ which is isomorphic to an affine $k$-variety. 
Since resolutions of singularities exist for all varieties in characteristic 0
by \cite{MR0199184}, we can find  a smooth projective compactification $\overline U$ of $U$ over $k$. Then $\overline U$ is rationally connected because it is birationally equivalent to $X$. 
So $\overline U(k)\neq\emptyset$ by assumption, and we conclude that  $X(k)\neq\emptyset$  by \cite{MR0095851}.
 \end{proof}
 
 If the base field is large in the sense of \cite[Proposition 1.1]{MR1405941}  we can remove also the properness assumption.
 \begin{lemma}\label{lem:conj_large}
 Let $k$ be a large field of characteristic 0 and $n$ a positive integer. If all smooth projective rationally connected $k$-varieties of dimension $n$ have  $k$-points, then every rationally connected $k$-variety of dimension $n$ has a $k$-point.
 \end{lemma}
 \begin{proof}
 For every  rationally connected $k$-variety $X$  we can find a smooth projective compactification $X'$ of a smooth affine open subset of $X$ as in the proof of Lemma \ref{lem:conj_sing}. Since $X'(k)\neq\emptyset$ by assumption, the set $X'(k)$ is dense in $X'$ as $k$ is large, hence $X(k)\neq\emptyset$.
 \end{proof}

In the proof of Theorem \ref{thm:equivalent_formulation} we use the following form of the Minimal Model Program, which is a consequence of \cite[Corollary 1.3.3]{MR2601039}. See \cite[Proposition 2.3]{MR3098794} for a proof. 

\begin{prop} \label{prop:building_block} 
Let $k$ be a field of characteristic 0. Let $X$ be a smooth projective rationally connected $k$-variety. 
Then there exist a birational map $X\dashrightarrow X'$ and a projective dominant morphism $f:X'\to Y$ such that 
$X'$ is a projective $\QQ$-factorial terminal $k$-variety, 
$Y$ is a  projective $\QQ$-factorial log terminal $k$-variety, 
$\dim Y<\dim X'$, 
$\rk\pic(Y)=\rk\pic(X')-1$,  
and the generic fiber of $f$ is a terminal Fano variety.
\end{prop}

We can now prove the main theorem of this section.
 \begin{proof}[Proof of Theorem \ref{thm:equivalent_formulation}]
For every positive integer $n$, the implication \ref{item:thm:equiv:ratconn} $\Rightarrow$ \ref{item:thm:equiv:fano} holds by Lemma \ref{lem:conj_sing} because every terminal Fano variety is rationally connected by \cite{MR2208131}.
 
To prove the reverse implication, we fix a positive integer $n$ and assume that \ref{item:thm:equiv:fano} holds.
Let $X$ be a smooth proper rationally connected $k$-variety of dimension $\leq n$. We prove that $X$ has a $k$-rational point by induction on the dimension of $X$.
 If $\dim X=0$, $X=\spec k$. 
  Assume that $\dim X\geq1$ and that 
 every smooth proper rationally connected $k$-variety of dimension $<\dim X$ has a $k$-point.
By Lemma \ref{lem:conj_sing} we can assume, without loss of generality, that $X$ is projective.
Let $f:X'\to Y$ be the fibration with $X'$ birationally equivalent to $X$ provided by Proposition \ref{prop:building_block}. 

We first prove that $X'(k)\neq\emptyset$.
If $\dim Y=0$, then $X'$ is a $\QQ$-factorial terminal Fano variety of Picard rank 1. 
So $X'(k)\neq\emptyset$ by \ref{item:thm:equiv:fano}. 
If $\dim Y>0$, then $Y$ is rationally connected and has dimension $<\dim X$. 
Hence,  $Y(k)\neq\emptyset$ by the induction hypothesis 
combined with Lemma \ref{lem:conj_sing}.
Let $y\in Y(k)$. By \cite[Theorem 1.2]{MR2491906} the fiber of $f$ over $y$ contains a  projective rationally connected $k$-subvariety $Z$.
Since $\dim Y>0$, $\dim Z<\dim X'=\dim X$. So $Z(k)\neq\emptyset$ by the induction hypothesis combined with Lemma \ref{lem:conj_sing}. Then $X'(k)\neq\emptyset$. 

 Let $x\in X'(k)$. 
 Let $h:X''\to X'$ be a resolution of singularities.  
 By \cite[Theorem 1.3]{MR2491906}, the fiber  of $h$ over $x$ contains a rationally connected $k$-subvariety $W$. 
 Since $h$ is a proper birational morphism,  $W$  is a proper rationally connected $k$-variety with $\dim W <\dim X''=\dim X$. Thus $ W(k)\neq\emptyset$ by the induction hypothesis combined with Lemma \ref{lem:conj_sing}. 
 Hence, $X''(k)\neq\emptyset$, and $X(k)\neq\emptyset$ by  \cite{MR0095851}. 
 \end{proof}

 \subsection{A consequence of boundedness of Fano varieties}
 
 We show that Birkar's boundedness of Fano varieties \cite{MR4224714} implies boundedness of Hilbert polynomials for terminal Fano varieties over nonclosed fields of characteristic 0.
 
 \begin{prop}\label{prop:boundedness}
  Let $k$ be a field of characteristic 0 and $n$  a positive integer. Then there exist a positive integer $N$ and finitely many polynomials $f_1,\dots,f_s\in\QQ[t]$ such that for every terminal Fano variety $X$ of dimension $n$ there exists an embedding $X\subseteq\PP^N_k$ such that $X$ has Hilbert polynomial $f_i$ for some $i\in\{1,\dots,s\}$.
 \end{prop}
 \begin{proof}
 Let $\overline k$ be an algebraic closure of $k$. 
By \cite[Theorem 1.1]{MR4224714} there are finitely many projective morphisms of $\overline k$-varieties, say $\{V^i\to T^i\}_{i\in I}$ for a finite set $I$, such that for every terminal Fano $k$-variety $X$ of dimension $n$ there exists $i\in I$, a point $t\in T^i(\overline k)$ and an isomorphism $V^i_t\cong X_{\overline k}$ of $\overline k$-varieties.  
Up to replacing each $T^i$ by a suitable finite stratification, we can assume that $T^i$ is smooth, $V^i\to T^i$ is flat, $K_{V^i}|_{V^i_t}=K_{V^i_t}$ for all $t\in T^i(\overline k)$ and there exists $m\in \ZZ_{>0}$ such that $-mK_{V^i}$ is Cartier and relatively very ample over $T^i$ for all $i\in I$. 
 We can use the complete linear system $|-mK_{V^i}|$ to embed $V^i$ into a projective space  $\PP^{N_i}_{T^i}$ over $T^i$ for all $i\in I$, where $N_i$ is the projective dimension of the linear system $|-mK_{V^i}|$ for all $i\in I$. 
Then for every $i\in I$ the Hilbert polynomial $f_i$ of $V^i_t$ with respect to the projective embedding given by $|-mK_{V^i_t}|$  is independent of the choice of $t\in T^i$ by \cite[Corollary III.9.13]{MR0463157}. 
Let $s$ be the cardinality of $I$ and write $I=\{1,\dots,s\}$. 
 Let $N=\max_{1\leq i\leq s} N_i$. Then for every terminal Fano variety $X$ of dimension $n$ over $k$, there exists $i\in\{1,\dots,s\}$ such that 
 $X$ has Hilbert polynomial $f_i$ under the embedding $X\subseteq\PP^{N_i}_k$ defined by
 the complete linear system $|-m K_X|$.
 Choose a linear embedding $\PP^{N_i}_k\subseteq \PP^N_k$. Then $X$ has Hilbert polynomial $f_i$ under the induced embedding $X\subseteq\PP^N_k$.
 \end{proof}
  
\subsection{Application to rationally connected varieties over $\QQ_p^{\nr}$}  
We apply the previous results to prove Corollary \ref{cor:duesler-knecht_enhanced}.
  
 \begin{prop}\label{prop:logic}
 Let $f\in\QQ[t]$ and $N\in\NN$. Then there exists a finite set $S(f,N)$ of prime numbers such that for every prime number $p\notin S(f,N)$, every projective rationally connected variety $X\subseteq\PP^N_{\QQ_p^{\nr}}$ with Hilbert polynomial $f$ has a $\QQ_p^{\nr}$-point.
 \end{prop}
 \begin{proof}
 Replace \cite[Theorem IV.3.11]{MR1440180} by \cite[Proposition 2.6]{MR3031573} in
the proof of \cite[Theorem 1.3]{MR3760308}.
 \end{proof}
 
 \begin{proof}[Proof of Corollary \ref{cor:duesler-knecht_enhanced}]
 Let $n$ be a positive integer. Combining Propositions \ref{prop:boundedness} and \ref{prop:logic} we obtain a finite set $S(n)$ of prime numbers  such that for every prime number $p\notin S(n)$ every terminal Fano variety of dimension $n$ over $\QQ_p^{\nr}$ has a $\QQ_p^{\nr}$-point. 
 Then by Theorem \ref{thm:equivalent_formulation} every smooth proper rationally connected variety of dimension $n$ over $\QQ_p^{\nr}$ with $p\notin S(n)$ has a $\QQ_p^{\nr}$-point. 
We conclude by Lemma \ref{lem:conj_large} as $\QQ_p$ and  $\QQ_p^{\nr}$ are large fields 
by the Implicit Function Theorem over local fields \cite[p.73]{MR2179691} and \cite[Proposition 1.2]{MR1405941}, respectively.
\end{proof}

\section{Del Pezzo surfaces}
\label{sec:dP_classification}

In this section we collect some properties related to the existence of rational points on rationally connected curves and surfaces. 
We recall that terminal varieties of dimension $\leq 2$  are smooth \cite[Corollary 5.18]{MR1658959}, and that the birational classification of smooth rationally connected curves and surfaces is well understood.

\begin{remark}
\label{rem:brauer1}
We recall that isomorphism classes of Severi-Brauer varieties of dimension $n$ over $k$ are in bijection with $(n+1)$-torsion elements of $\Br(k)$  \cite[Theorem 5.2.1]{MR2266528},  and the projective space corresponds to the neutral element.
Moreover, a Severi-Brauer variety over $k$ is a projective space if and only if it has a $k$-rational point \cite[Theorem 5.1.3]{MR2266528}. In particular, a Severi-Brauer variety of dimension $n$ acquires a rational point over a base field extension of degree at most $n+1$. 
\end{remark}

Smooth curves are rationally connected if and only if they are Fano. Hence, terminal rationally connected curves are precisely the smooth conics. 
\begin{lemma}[{\cite[\S1.3, Example 5.2.4]{MR2266528}}]
For a field $k$, the following are equivalent:
 \begin{enumerate}
  \item $\Br(k)$ has nontrivial $2$-torsion;
  \item there is a conic $C$ over $k$ with $C(k)=\emptyset$.
 \end{enumerate}
 \end{lemma}

Smooth surfaces are rationally connected if and only if they are rational. Hence, terminal rationally connected surfaces are rational conic bundles or del Pezzo surfaces \cite{MR525940}.

\begin{lemma}
\label{lem:conicbundle}
Let  $X$ be a rational conic bundle surface over a field $k$. Then
\begin{enumerate}[label=(\roman*), ref=(\roman*)]
\item 
\label{item:conicbundle_general}
there exists a finite extension $k'$ of $k$ of degree at most $4$ such that $X$ has a $k'$-point.
\item 
\label{item:conicbundle_example}
There exists a rational conic bundle over $\QQ$ that has no rational points on any finite extension of $\QQ$ of degree at most $3$.
\end{enumerate}
\end{lemma}
\begin{proof}
The conic bundle structure is given by a morphism $f:X\to C$ where $C$ and the fibers of $f$ are conics. Then $C$ acquires  a rational point over a suitable quadratic extension of $k$ and the fiber of $f$ over that point acquires a rational point over a suitable further quadratic extension.
Moreover, given two conics $C_1$ and $C_2$ without rational points over $\QQ$ and with distinct splitting fields, for example, $C_1=\{x^2+y^2+z^2=0\}$ and $C_2=\{x^2+2y^2+5z^2=0\}$ by \cite[Example III.2.13]{MR2104929},
then $C_1\times C_2$ has no rational points on any finite extension of $\QQ$ of degree $\leq 3$.
\end{proof}

Del Pezzo surfaces are the Fano varieties of dimension $2$, and they are classified according  their degree, i.e., the self intersection number of the canonical class, which is an integer between $1$ and $9$. 

\begin{lemma}
\label{lem:delPezzo}
Let  $X$ be a del Pezzo surface of degree $d$ over a field $k$. Then
\begin{enumerate}[label=(\roman*), ref=(\roman*)]
\item 
\label{item:dPgeneral}
$X$ acquires a rational point over a suitable finite extension of $k$ of degree at most $d$.
\item 
\label{item:dP9}
If $d=9$, then $X$ acquires rational points over a suitable finite extension of $k$ of degree at most $3$.
\item 
\label{item:dP8}
If $X_{\overline k}\cong \PP^1_{\overline k}\times \PP^1_{\overline k}$, then $X$ acquires rational points over a suitable finite extension of $k$ of degree at most $4$.
\item 
\label{item:dP7}
\label{item:dP5}
If $d=5$, or $d=7$, or $d=8$ and $X_{\overline k}\not\cong \PP^1_{\overline k}\times \PP^1_{\overline k}$, then $X$ is $k$-rational and $X(k)\neq\emptyset$.
\item 
\label{item:dP6}
There exists a del Pezzo surface of degree $6$ over $\QQ$ that has no rational points on any finite extension of $\QQ$ of degree smaller than $6$.
\end{enumerate}
\end{lemma}
\begin{proof}
Part \ref{item:dPgeneral} is a consequence of the fact that the base locus of the anticanonical linear system has dimension at most $0$.
In \ref{item:dP9} $X$ is a Severi-Brauer surface, and we conclude  by Remark \ref{rem:brauer1}.
Part \ref{item:dP8} follows from Lemma \ref{lem:conicbundle} if $\rk\pic(X)=2$ and from Lemma \ref{lem:splitting_products_proj_spaces} if $\rk\pic(X)=1$.
Part \ref{item:dP7} follows from \cite{MR0376684, MR1260765} for $d=5$, and from \cite[Theorem (3.7)]{MR225780} together with \cite[\S1.4]{MR3114932} for the other cases.
 
For  \ref{item:dP6},  let $K=\QQ(\sqrt {29})$ and  $L=\QQ(\zeta_7+\zeta_7^{-1})$, where $\zeta_7$ is a primitive $7$th root of unity. Let $E$ be the smallest number field containing both $K$ and $L$. Then $K$, $L$ and $E$ are cyclic Galois extensions of $\QQ$ of degree $2$, $3$ and $6$, respectively. Let $b=(\sqrt {29}-5)/2$  and $q=(\zeta_7+\zeta_7^{-1})^2-2$.
For $F\in\{K, L, \QQ\}$ we denote by $N_{E/F}$ the norm of $E$ over $F$ and by $\Br(E/F)$ the Brauer group of $F$ relative to $E$. Since the extensions are cyclic, by \cite[Corollary 4.4.10]{MR2266528} we can write $\Br(E/F)=F^\times/N_{E/F}(E^\times)$. Under this identification, the corestriction maps $\Br(E/K)\to \Br(E/\QQ)$ and $\Br(E/L)\to \Br(E/\QQ)$ are induced by the norms $N_{K/\QQ}$ and $N_{L/\QQ}$, respectively, see \cite[Example 99.6]{MR2427530}. Let $B\in \Br(E/K)$ be the class of $b$ and $Q\in \Br(E/L)$ be the class of $q$. 
Computations (e.g., via the  software Sage \cite{sage})
show that 
\[
b\notin  N_{E/K}(E^\times), \quad q\notin  N_{E/L}(E^\times), \quad N_{K/\QQ}(b), N_{L/\QQ}(q)\in N_{E/\QQ}(E^\times).
\]
Let $X$ be the del Pezzo surface of degree 6 associated to $(B,Q,E)$ by \cite[Theorem 3.4]{MR2564828}. By \cite[Corollary 3.5]{MR2564828} $X$  has a rational point on an extension $k'$ of $\QQ$ if and only if both $B$ and $Q$ are split over $k'$. This happens if and only if $E\subseteq k'$, which implies $[k':\QQ]\geq 6$.  
\end{proof}

\section{Fano threefolds}
\label{sec:fano_3folds_existence}
We recall the classification of terminal Fano threefolds that are mentioned in Conjecture \ref{conj:terminal_Fano_final} and in Theorem \ref{thm:equivalent_formulation}\ref{item:thm:equiv:fano}, and we investigate some of their properties related to the existence of rational points for the proof of Theorem \ref{thm:extension_points_effective}.

Unless stated otherwise, in this section $X$ denotes a terminal Fano threefold over a field $k$ of characteristic $0$. We distinguish between Gorenstein (for example, smooth) and non-Gorenstein varieties.  We recall from \S \ref{intro:sec:classification} that Fano varieties are classified according to the index $r(X_{\overline k})$, which is a positive rational number up to $\dim X+1$, see \cite[Corollary 2.1.13]{MR1668579}. In the Gorenstein case $r(X_{\overline k})$ is a positive integer. We denote by $H$ a fundamental divisor of $X_{\overline k}$, i.e.,  $H\in\pic(X_{\overline k})$ such that $-K_{X_{\overline k}}=r(X_{\overline k})H$.

The classification of Fano threefolds uses the following invariants:
\begin{itemize}
\item the geometric Picard rank $\rk(\pic(X_{\overline k}))$;
\item the index $r(X_{\overline k})$;
\item the degree $H^3$, i.e., the top self intersection of the fundamental divisor $H$.
\end{itemize}
We recall that $(-K_{X_{\overline k}})^3=r(X_{\overline k})^3 H^3$.

We consider first Gorenstein Fano threefolds, which are well studied and mostly classified, while the classification of the non-Gorenstein ones is very far from being complete.

\begin{lemma}
\label{lem:basepoints}
Let $X$ be a Gorenstein terminal Fano threefold over a field $k$ of characteristic $0$.
\begin{enumerate}[label=(\roman*), ref=(\roman*)]
\item 
\label{item:lem:basepoints:baselocus}
If the base locus of $|-K_X|$ is nonempty, then $r(X_{\overline k})=1$ and $X$ acquires a rational point over a suitable base field extension of degree at most 2.
\item 
\label{item:lem:basepoints:basepointfree}
If $|-K_X|$ is base point free, then 
 $X$ acquires a rational  point over a suitable base field extension of degree at most $(-K_X)^3$.
\end{enumerate}
\end{lemma}
\begin{proof}
Part \ref{item:lem:basepoints:baselocus} follows from the fact that if nonempty, the base locus of $|-K_X|$ either consists  of one $k$-point or it is a smooth conic by \cite[Theorem 0.5]{MR1030501}. Part \ref{item:lem:basepoints:basepointfree} is a consequence of Bertini's theorem.
\end{proof}

The rest of the section is devoted to improve the bound in Lemma \ref{lem:basepoints}\ref{item:lem:basepoints:basepointfree}. 
We start by collecting a few facts about indices of Fano varieties.
 \begin{lemma}\label{lem:index}
 Let $X$ be a log terminal Fano variety over a field $k$ of characteristic 0. Let $\overline k$ be a separable closure of $k$. Let $H\in\pic(X_{\overline k})$ such that $-K_{X_{\overline k}}=r(X_{\overline k})H$. Then $H\in\pic(X_{\overline k})^{\gal(\overline k/k)}$. Moreover, 
\begin{enumerate}[label=(\roman*), ref=(\roman*)]
\item  $I(X)=I(X_{\overline k})$ and $q(X)\mid q(X_{\overline k})$;
\item  if $\Br(k)=0$ or if $X(k)\neq\emptyset$, then $r(X)=r(X_{\overline k})$.
\end{enumerate} 
 \end{lemma}
 \begin{proof}
 Let $q:=q(X_{\overline k})$ and $I:=I(X_{\overline k})$, so that $r(X_{\overline k})=q/I$.
  Since the canonical divisor class is invariant under field extension, $IK_X$ is a Cartier divisor class on $X$. Thus $I(X)=I$ and $qH=-IK_X\in\pic(X)$.
Let $g\in\gal(\overline k/k)$. Since $qH=gqH=qgH$, the element $H-gH$ of $\pic(X_{\overline k})$ is torsion. But $\pic(X_{\overline k})$ is free  by \cite[Proposition 2.1.2]{MR1668579}. Hence, $g H=H$. Thus $H$ is invariant under the action of $\gal(\overline k/k)$ on $\pic(X_{\overline k})$. 
 The divisibility condition $q(X)\mid q$ comes from the definition of  $q(X)$ and $q(X_{\overline k})$ and the fact that $\pic(X_{\overline k})$ is free.
 If $\Br(k)=0$ or $X(k)\neq\emptyset$,
  the exact sequence \cite[(1.5.0)]{MR899402}
 gives $\pic(X)=\pic(X_{\overline k})^{\gal(\overline k/k)}$.
 \end{proof}
\begin{lemma}\label{lem:split_line_bundle}
Let $X$ be a proiective geometrically integral variety over a field $k$. Let $\overline k$ be a separable closure of $k$. Let $H\in\pic(X_{\overline k})^{\gal(\overline k/k)}$. Then there exists a field extension $k'/k$ of degree at most $\dim_{\overline k}H^0(X_{\overline k},\mathcal{O}_{X_{\overline k}}(H))$ such that $H\in\pic(X_{k'})$.
\end{lemma}
\begin{proof}
By \cite[(7.4), Aside 32]{arXiv:1606.04368v1}, the exact sequence 
\cite[Proposition 69]{arXiv:1606.04368v1} associates to $H$ a Severi-Brauer $k$-variety $P$ of dimension $\dim_{\overline k}H^0(X_{\overline k},\mathcal{O}_{X_{\overline k}}(H))-1$
such that $P$ splits over a finite extension $k'$ of $k$ contained in $\overline k$ if and only if $H\in\pic(X_{k'})$. 
 Then we conclude by Remark \ref{rem:brauer1}.
\end{proof}

\subsection{Gorenstein of geometric Picard rank $1$}
\label{sec:smooth_picardrk1}

In this subsection let $X$ be a Gorenstein terminal Fano threefold of geometric Picard rank $1$ over  a field $k$ of characteristic $0$. 

\subsubsection{Index  $4$} If $r(X_{\overline k})=4$, then $H^3=1$, $(-K_{X_{\overline k}})^3=64$, and $X$ is a Severi Brauer variety of dimension $3$. Hence, $X$ is smooth and acquires a rational point over a suitable field extension of degree at most $4$ by Remark \ref{rem:brauer1}.

\subsubsection{Index $3$} If $r(X_{\overline k})=3$, then $H^3=2$ and $(-K_{X_{\overline k}})^3=54$.
\begin{lemma}
Let $X$ be a Gorenstein terminal Fano threefold  over  a field $k$ of characteristic $0$. Assume that  $\rk(\pic(X_{\overline k}))=1$ and $r(X_{\overline k})=3$. Then $X$ is a quadric hypersurface in $\PP^4_k$. In particular, $X$ acquires a rational point over a suitable base field extension of degree at most $2$.
\end{lemma}
\begin{proof}
By \cite[Theorem 0]{MR664549} (cf.~\cite[Theorem 3.1.14]{MR1668579}) we know that $X_{\overline k}\stackrel{|H|}{\lhook\joinrel\longrightarrow}\PP^4_{\overline k}$ is a quadric hypersurface. By Lemma \ref{lem:index}, $H$ is  invariant under the $\gal(\overline k/k)$ action on $\pic(X_{\overline k})$.
Hence, by \cite[Theorem 3.4]{MR3644253} there is a Severi-Brauer $k$-variety $P$ of dimension $4$ and a morphism $X\to P$ that is a $k$-model of the inclusion $X_{\overline k}\subseteq\PP^4_{\overline k}$. Let $\alpha_H$ be the image of $H$ in $\Br(k)$ under the morphism $\alpha$ from \cite[(69.1)]{arXiv:1606.04368v1}. Since $X_{\overline k}$ is a quadric in $\PP^4_{\overline k}$, we see that $2\alpha_H=0$, but also $5\alpha_H=0$ (e.g., by \cite[Proposition 44]{arXiv:1606.04368v1}), from which we conclude that $\alpha_H=0$ and  $P\cong \PP^4_{k}$. Then $X$ is a quadric hypersurface in $\PP^4_{k}$.
\end{proof}

\subsubsection{Index $2$} If $r(X_{\overline k})=2$, then $H^3$ is an integer between $1$ and $5$, and $(-K_{X_{\overline k}})^3=8 H^3$. We recall the classification from \cite[Corollary 0.8]{MR1030501} (cf.~\cite[Theorems 3.2.5 and 3.3.1]{MR1668579} and \cite[Theorem 1.7]{MR3100917}).
\begin{itemize}
\item If $H^3=1$, $X_{\overline k}$ is a sextic hypersurface in the weighted projective space $\PP_{\overline k}(3,2,1,1,1)$.
\item If $H^3=2$, $X_{\overline k}\stackrel{|H|}{\longrightarrow}\PP^3_{\overline k}$ is a double cover ramified along a smooth quartic surface.

\item If $H^3=3$, $X_{\overline k}\stackrel{|H|}{\lhook\joinrel\longrightarrow}\PP^4_{\overline k}$ is a smooth cubic hypersurface.

\item If $H^3=4$, $X_{\overline k}\stackrel{|H|}{\lhook\joinrel\longrightarrow} \PP^5_{\overline k}$ is a smooth complete intersection of two quadrics.

\item If $H^3=5$, $X_{\overline k}\stackrel{|H|}{\lhook\joinrel\longrightarrow} \PP^6_{\overline k}$ is  an intersection of five quadrics. If $X$ is smooth, $X_{\overline k}$ is a section of the Grassmannian $\Gr(2,5)\subseteq\PP^9_{\overline k}$ by a linear space of codimension $3$ under the Plücker embedding.
\end{itemize}

\begin{lemma}
Let $X$ be a Gorenstein terminal Fano threefold  over  a field $k$ of characteristic $0$. Assume that  $\rk(\pic(X_{\overline k}))=1$ and $r(X_{\overline k})=2$. Then $X$ acquires a rational point over a suitable field extension of degree at most $d_1(H^3)$, for
\[
(H^3, d_1(H^3)) \ \in\  \{(1,1), \ (2,8), \ (3, 3), \ (4, 16), \ (5,1)\}.
\]
\end{lemma}
\begin{proof}
By \cite[Remarks 3.2.2(ii)]{MR1668579} we know that $\dim_{\overline k}H^0(X_{\overline k},H)=H^3+2$. 
Hence, by  Lemmas \ref{lem:index} and \ref{lem:split_line_bundle} there exists a finite extension $k'$ of $k$ of degree at most $H^3+2$ such that $H\in\pic(X_{k'})$. By \cite[Proposition 3.2.3]{MR1668579} the general member of $|H|$ is a del Pezzo surface of degree $H^3$ over $k'$, and hence acquires a rational point over a further extension of degree $H^3$ by Lemma \ref{lem:delPezzo}\ref{item:dPgeneral}. If $H^3=5$, then $X_{k'}$ has a $k'$-point by Lemma \ref{lem:delPezzo}\ref{item:dP5}.

Let $\alpha_H$ be the image of $H$ in $\Br(k)$ under the morphism $\alpha$ from \cite[(69.1)]{arXiv:1606.04368v1}. Since $2 H=-K_X\in\pic(X)$, we have $2\alpha_H=0$ in $\Br(k)$. But we have also $(H^3+2)\alpha_H=0$ (e.g., by \cite[Proposition 44]{arXiv:1606.04368v1}). Hence, if $H^3\in\{1,3,5\}$, we conclude that  $\alpha_H=0$, and we can choose $k'=k$. If $H^3=4$,  we can choose $k'$ of degree at most $4$ over $k$ by \cite[Theorem 53, Corollary 54]{arXiv:1606.04368v1}.
\end{proof}

\subsubsection{Smooth of index $1$}
\label{sec:smooth_mukai}
  If $X$ is smooth and $r(X_{\overline k})=1$, then $H=-K_X\in\pic(X)$, and $H^3$ is an even integer between $2$ and $22$, $H^3\neq 20$. Let $g:=(-K_X)^3/2+1$. Then $g$ is an integer between $2$ and $12$, $g\neq 11$. According to \cite{MR1944132}, $|H|$ is base point free and it induces a morphism $X\to\PP^{g+1}_k$.
We recall the classification from \cite[Theorem 1.10]{MR1944132}, \cite[Theorem 5.2.3, or \S12.2]{MR1668579}. See \cite{forms_ci} for the Galois descent of complete intersections.
\begin{itemize}
\item If $H^3=2$, then $g=2$ and $X\stackrel{|H|}{\longrightarrow}\PP^3_k$ is a double cover ramified along a sextic surface.

\item If $H^3=4$, then  $g=3$ and $X\stackrel{|H|}{\longrightarrow} \PP^4_k$ is either an embedding as quartic hypersurface, or a double cover of a quadric hypersurface in $\PP^4_k$.
\end{itemize}

If $g\geq 4$, $H$ is very ample.

\begin{itemize}
\item If $H^3=6$, then  $g=4$ and $X\stackrel{|H|}{\lhook\joinrel\longrightarrow}\PP^5_k$ is a complete intersection of a quadric and a cubic.

\item If $H^3=8$, then  $g=5$ and $X\stackrel{|H|}{\lhook\joinrel\longrightarrow}\PP^6_k$ is a complete intersection of three quadrics.

\item If $H^3=10$, then  $g=6$ and $X\stackrel{|H|}{\lhook\joinrel\longrightarrow}\PP^7_k$ is a $k$-form of a complete intersection of a cone over the Grassmannian $\Gr(2,5)$ under the Pl\"ucker embedding with a quadric and a linear space of codimension $3$ (cf.~\cite[Examples 5.2.2(i)]{MR1668579}).
\end{itemize}

If $g\geq7$,  a general hyperplane section of the embedding $X\stackrel{|H|}{\lhook\joinrel\longrightarrow}\PP^{g+1}_k$ cuts a  smooth K3 surface $S$ in $X$ by \cite[\S7.7]{MR1492525}  and \cite[Proposition 7.8]{MR1944132}.

\begin{itemize}
\item If $H^3=12$, then  $g=7$. Denote by $\mathcal{N}_{S/\PP^7_k}^*$  the dual of the normal bundle of $S\subseteq\PP^7_k$.
In the proof of \cite[Theorem 4.7]{MR1944132}, the vector bundle $\mathcal{N}_{S/\PP^7_k}^*\otimes\OO_S(2)$ determines a $\gal(\overline k/k)$-equivariant embedding $X_{\overline k}\subseteq\Gr(5,10)$ whose image is a linear section of an orthogonal Grassmannian $\ortgrass(5,10,Q)\subseteq\PP^{15}_{\overline k}$, where $Q$ is a $\gal(\overline k/k)$-invariant quadric  hypersurface $Q$ in $\PP^9_{\overline k}$ by a version of \cite[Corollary 2.5]{MR1363081} for $S$.
\end{itemize}

If $g\geq 8$, a general hyperplane section of $S\stackrel{|H|_S|}{\lhook\joinrel\longrightarrow}\PP^g_k$ is a smooth curve of genus $g$ which determines a Mukai-Lazarsfeld bundle $E$ on $S$ (cf.~\cite[\S1.3]{MR3051369}). 

\begin{itemize}
\item If $H^3=14$,  then $g=8$ and $E$ induces a $\gal(\overline k/k)$-equivariant embedding of $X_{\overline k}$ in $\Gr(2,6)$ such that if  $\Gr(2,6)\subset\PP^{14}_{\overline k}$ is the Pl\"ucker embedding, $X_{\overline k}$ is a linear section of $\Gr(2,6)$ by  a $\gal(\overline k/k)$-invariant linear subspace of codimension $5$.

\item If $H^3=16$,  then $g=9$ and $E$ induces a $\gal(\overline k/k)$-equivariant embedding of $X_{\overline k}$ in $\Gr(3,6)$ such that if  $\Gr(3,6)\subseteq\PP^{19}_{\overline k}$ is the Pl\"ucker embedding, $X_{\overline k}$ is a complete intersection of $\Gr(3,6)$ with the zero locus of a $\gal(\overline k/k)$-invariant global section of the second exterior power of the dual of the tautological bundle of $\Gr(3,6)$ and a $\gal(\overline k/k)$-invariant linear subspace of codimension $3$ (cf.~\cite[Example 5.1]{MR1944132}). 

\item If $H^3=18$,  then $g=10$ and $E$ induces a $\gal(\overline k/k)$-equivariant embedding of $X_{\overline k}$ in $\Gr(5,7)$ such that if $\Gr(5,7)\subseteq\PP^{20}_{\overline k}$ is the Pl\"ucker embedding, $X_{\overline k}$ is a complete intersection of $\Gr(5,7)$ with the zero locus of a $\gal(\overline k/k)$-invariant global section of the fourth exterior power of the dual of the tautological bundle of $\Gr(5,7)$ and a $\gal(\overline k/k)$-invariant linear subspace of codimension $2$ (cf.~\cite[Example 5.2]{MR1944132}). 

\item If $H^3=22$,  then $g=12$ and $E$ induces a $\gal(\overline k/k)$-equivariant embedding of $X_{\overline k}$ in $\Gr(3,7)$ such that 
$X_{\overline k}$ is the zero locus in $\Gr(3,7)$ of three linearly independent $\gal(\overline k/k)$-invariant global sections of the second exterior power of the dual of the tautological bundle of $\Gr(3,7)$ (cf.~\cite[\S5, p.15]{MR1944132}). 
\end{itemize}

Now we investigate some properties of linear sections of orthogonal Grassmannians that we can use  to study rational points in the case $g=7$.

Given a nonsingular quadric hypersurface $Q\subseteq\PP^9$, 
we denote by $\ortgrass(5,10,Q)$ the orthogonal Grassmannian of isotropic 4-dimensional linear subspaces of $\PP^9$. It is also known as spinor variety in $\PP^{15}$. We refer to \cite{MR1780430} and \cite{MR3833474} for a detailed description.

\begin{lemma}\label{lem:mukai_7}
 Let $k$ be a  field of characteristic $0$ and let $\overline k$ be a separable closure of $k$. 
Let $Q\subseteq\PP^9_k$ be a nonsingular quadric hypersurface. Let $Y\subseteq \PP^{15}_k$ be a $k$-model of the spinor embedding $\ortgrass(5,10,Q)\subseteq \PP^{15}_{\overline k}$. 
Let $L\subseteq\PP^{15}_k$ be a linear subspace of codimension $7$ such that $X:=L\cap \ortgrassfield$ is geometrically irreducible of dimension $3$. 
Then $X$ acquires a rational point over a suitable quadratic extension of $k$.
 \end{lemma}
\begin{proof}
We denote by ${\PP^{15}_k}^*$ the dual projective space as in \cite[Notation 3.1]{MR3833474} and by $Y^*\subseteq {\PP^{15}_k}^*$ the dual of $Y\subseteq \PP^{15}_k$. The dual $L^*$ of $L$ is a $6$-dimensional linear subspace of ${\PP^{15}_k}^*$, hence $L^*\not\subseteq Y^*$ by \cite[Theorem 3.2]{MR3833474}. 
Let $v$ be a $k$-point on $L^*$ that is not contained in $Y^*$, and let $L_v\subseteq\PP^{15}_k$ be the corresponding hyperplane.
By \cite[Lemma 5.10]{MR3833474} the intersection $L_v\cap Y$ contains a quadric $Q_v$ of dimension $6$ defined over $k$. Then $L\cap Q_v$ is a nonempty $k$-subvariety of degree $2$ in $L$, it is contained in $X$, and it acquires a rational point over a suitable extension of $k$ of degree at most $2$.
\end{proof}

\begin{lemma}\label{lem:mukai_12}
Let $X$ be a smooth Fano threefold  over  a field $k$ of characteristic $0$. Assume that $\rk(\pic(X_{\overline k}))=1$, $r(X_{\overline k})=1$ and $H^3=22$. 
Then $X$ acquires a rational point over a suitable field extension of degree at most $6$.
\end{lemma}

\begin{proof}
The Fano variety of conics of $X$ is a Severi Brauer surface by \cite[Proposition B.4.1]{MR3776469}. 
Hence,  $X$ contains smooth conics defined over a cubic extension of $k$ by Remark \ref{rem:brauer1},  and has rational points over a further quadratic extension.
\end{proof}

\subsubsection{Singular Gorenstein of index $1$}
\label{sec:sing_gor_rkpic_1}

\begin{defin}\label{def:decomposable}
The anticanonical linear system $|-K_X|$ of a Gorenstein Fano variety $X$ of index 1 is said to have a movable decomposition if there are Weil divisors $A$ and $B$ such that the linear systems $|A|$ and $|B|$ have positive dimension and $A+B$ is linearly equivalent to $-K_X$. If $|-K_{X_{\overline k}}|$ has no movable decomposition, for an algebraic closure $\overline k$ of the base field $k$, we say that $X$ is indecomposable. Otherwise, we say that $X$ is decomposable.
\end{defin}

Indecomposable Gorenstein terminal Fano threefolds of index 1 and geometric Picard rank 1 are classified in \cite[Theorems 1.10, 6.5]{MR1944132} 
(we remark that \cite[Proposition 7.8]{MR1944132} and \cite[Theorem 1, Theorem 2.4]{MR1675146} show that the classification is exaustive), and the bounds from Section \ref{sec:smooth_mukai} apply.

Decomposable terminal Gorenstein Fano threefolds of index 1 and geometric Picard rank 1 are not completely classified. However,
by \cite[Proposition 4.1.12]{MR1668579}, \cite[Theorem 1.5]{MR2136260} (cf.~\cite[Proposition 6.1.1]{MR3711004}) we know that the cases with $g\in\{2,3,4,5\}$ are completely classified in \cite[Theorem 6.5]{MR1944132}. 

\begin{lemma}
Let $X$ be a Gorenstein terminal $\QQ$-factorial  Fano threefold over a field $k$ of characteristic $0$. Assume that $\rk(\pic(X_{\overline k}))=1$ and $r(X_{\overline k})=1$. Then $X$ acquires a rational point over a finite extension of $k$ of degree at most $18$.
\end{lemma}
\begin{proof}
By \cite{MR1489117, MR2860279},  $X_{\overline k}$ admits a smoothing that preserves the Picard group and the degree $(-K_{X_{\overline k}})^3$. Hence, $(-K_{X})^3\in\{2,4,6,8,10,12,14,16,18,22\}$.
If $(-K_{X})^3=22$, the only decomposable cases we are interested in are classified in \cite[Theorem 1.3]{MR3535377}, and their singular locus consists of a rational point. We conclude by Lemmas \ref{lem:basepoints} and \ref{lem:mukai_12}.
\end{proof}

\subsection{Gorenstein of geometric Picard rank $>1$}
\label{sec:rkpic>1}

In this subsection let $X$ be a Gorenstein terminal
Fano threefold of geometric Picard rank $>1$ and Picard rank 1 over  a field $k$ of characteristic $0$. 

\subsubsection{Products of projective spaces}

\begin{lemma}
\label{lem:splitting_products_proj_spaces}
Let $k$ be a field of characteristic $0$, and let $X$ be a twisted form of $\prod_{i=1}^n\PP^m_k$ over $k$ such that $\rk(\pic(X))=1$. 
Then there is an extension $\tilde k$ of $k$ of degree at most $n!(m+1)$ such that $X_{\tilde k}\cong \prod_{i=1}^n\PP^m_{\tilde k}$.
\end{lemma}
\begin{proof}
Let $e_1,\dots,e_n$ be the standard generators of $\pic(X_{\overline k})$. 
From the exact sequence
\[
1\to \prod_{i=1}^n\Aut(\PP^m_k)\to\Aut(\prod_{i=1}^n\PP^m_k)\to S_n\to 1
\]
and the induced exact sequence in \'etale cohomology,
there exists a unique finite Galois extension $k'$ of $k$ of degree between $n$ and $n!$ such that $e_i$ is $\gal(\overline k/k')$-invariant for all $i\in\{1,\dots,n\}$.
For every $i\in\{1,\dots,n\}$, let $\alpha_i$ be the image of $e_i$ in $\Br(k')$ under the morphism $\alpha$ from \cite[(69.1)]{arXiv:1606.04368v1}. Then $\alpha_1,\dots,\alpha_n$ form an orbit under the conjugation action of $\gal(k'/k)$ (see \cite[Construction 3.3.12]{MR2266528}). Hence, they split over the same field extension $\tilde k$ of $k'$, which has degree at most $m+1$ by Lemma \ref{lem:split_line_bundle}.
In particular, $e_1,\dots,e_n\in\pic(X_{\tilde k})$, and the morphism $X_{\tilde k}\to \prod_{i=1}^n\PP^m_{\tilde k}$ induced by the product of the projections corresponding to $e_1,\dots,e_n$ is an isomorphism.
\end{proof}

\subsubsection{Smooth}
\label{sec:smooth_picardrk>1}
Let $X$ be a smooth Fano threefolds of geometric Picard rank $>1$ and Picard rank 1. 
We recall the classification from \cite[Theorem 1.2, \S2]{MR3100918}. 

If $\rk\pic(X_{\overline k})=2$, we have
\begin{itemize}
\item {\cite[Case 1.2.1 a), p. 421]{MR3100918}}:  $r(X_{\overline k})=1$, $H^3=12$, and $X_{\overline k}$ is a divisor of bidegree $(2,2)$ in $\PP^2_{\overline k}\times\PP^2_{\overline k}$.
\item {\cite[Case 1.2.1 b), p. 421]{MR3100918}}:  $r(X_{\overline k})=1$, $H^3=12$, and $X$ is a  double cover $\pi:X\to V$ of a variety $V$ that belongs to Case (1.2.4) with branch locus a member of $|-K_V|$. 
\item {\cite[(1.2.2)]{MR3100918}}: $r(X_{\overline k})=1$, $H^3=20$, and $X_{\overline k}$ is a complete intersection of three divisors of bidegree $(1,1)$ in $\PP^3_{\overline k}\times\PP^3_{\overline k}$.
\item {\cite[(1.2.3)]{MR3100918}}: $r(X_{\overline k})=1$, $H^3=28$, and $X_{\overline k}$ is a blow up of a quadric in $\PP^4_{\overline k}$ along a twisted quartic curve.
\end{itemize}
If $\rk\pic(X_{\overline k})=3$, we have
\begin{itemize}
\item {\cite[(1.2.4)]{MR3100918}}: $r(X_{\overline k})=2$, $H^3=6$, and $X$ is a $k$-model of a divisor of bidegree $(1,1)$ in $\PP^2_{\overline k}\times\PP^2_{\overline k}$.
\item {\cite[(1.2.5)]{MR3100918}}: $r(X_{\overline k})=1$, $H^3=12$, and $X_{\overline k}$ is a double cover of $\PP^1_{\overline k}\times\PP^1_{\overline k}\times\PP^1_{\overline k}$ ramified along an element of $|-K_{\PP^1_{\overline k}\times\PP^1_{\overline k}\times\PP^1_{\overline k}}|$.
\item {\cite[(1.2.6)]{MR3100918}}: $r(X_{\overline k})=1$, $H^3=30$, and $X_{\overline k}$ is a complete intersection of three divisors in $\PP^2_{\overline k}\times\PP^2_{\overline k}\times\PP^2_{\overline k}$ of tridegrees $(0,1,1)$, $(1,0,1)$, $(1,1,0)$, respectively.
\item {\cite[(1.2.7)]{MR3100918}}: $r(X_{\overline k})=2$, $H^3=6$, and $X_{\overline k}\cong \PP^1_{\overline k}\times\PP^1_{\overline k}\times\PP^1_{\overline k}$.
\end{itemize}
If $\rk\pic(X_{\overline k})=4$, we have
\begin{itemize}
\item \emph{\cite[(1.2.8)]{MR3100918}}: $r(X_{\overline k})=1$, $H^3=24$, and $X$ is $k$-model of a divisor of multidegree $(1,1,1,1)$ in $\PP^1_{\overline k}\times\PP^1_{\overline k}\times\PP^1_{\overline k}\times\PP^1_{\overline k}$.
\end{itemize}

\begin{lemma}
Let $n\in\{1,\dots,8\}$, and let $X$ be a smooth Fano threefold over a field $k$ of characteristic $0$ as in \cite[(1.2.$n$)]{MR3100918}. Then $X$ acquires a rational point over a suitable base field extension of degree at most $d_2(n)$, for
\[
(n,d_2(n))\in\{(2,8 ), \ (3, 4), \ (4, 6),  \ (6, 18), \ (7, 12)\}.
\]
\end{lemma}

\begin{proof}
If $n=2$, we have 
$\pic(X_{\overline k})\cong\pic(\PP^3_{\overline k}\times\PP^3_{\overline k})$ by  \cite[Corollary IV.3.3]{MR0282977} and $-K_X$ is an element of bidegree $(1,1)$. 
By Lemma \ref{lem:splitting_products_proj_spaces} there is an extension $k'$ of $k$ of degree at most $8$ such that $X_{k'}\subseteq\PP^3_{k'}\times\PP^3_{ k'}$ is defined by three bihomogeneous polynomials $f_1,f_2,f_3$ of bidegree $(1,1)$  with coefficients in $k'$ in a set of variables $(x_0,x_1,x_2,x_3;y_0,y_1,y_2,y_3)$. Evaluating $f_1,f_2, f_3$ at $y_0=y_1=y_2=y_3=1$ gives three linear forms in $4$ variables, which have a nontrivial common zero $(\tilde x_0,\tilde x_1,\tilde x_2,\tilde x_3)\neq(0,0,0,0)$ over $k'$. Then $((\tilde x_0:\tilde x_1:\tilde x_2:\tilde x_3), (1:1:1:1))$ is a $k'$-point on $X$.

For $n=3$, 
let $\varphi: X_{\overline k}\to \PP^4_{\overline k}$ be the morphism induced by the blow-up and let $H$ be the hyperplane class in $\pic(\PP^4_{\overline k})$.
By \cite[pp.426-427,432]{MR3100918} there is a quadratic extension $k'$ of $k$ such that $\varphi^*(H)$ is invariant under the $\gal(\overline k/k')$ action on $\pic(X_{\overline k})$. By \cite[Theorem 3.4]{MR3644253} there is a Severi-Brauer $k'$-variety $P$ of dimension $4$ and a morphism $\varphi':X_{k'}\to P$ that is a $k'$-model of $\varphi$. Let $\alpha_H$ be the image of $H$ in $\Br(k')$ under the morphism $\alpha$ from \cite[(69.1)]{arXiv:1606.04368v1}. Since the image of $\varphi$ is a quadric in $\PP^4_{\overline k}$, we see that $2\alpha_H=0$, but also $5\alpha_H=0$, from which we conclude that $\alpha_H=0$ and  $P\cong \PP^4_{k'}$. Then $X_{k'}$ is the blow up of a quadric $Q$ in $\PP^4_{k'}$ along a suitable quartic curve. In particular, there is a suitable quadratic extension $k''$ of $k'$ such that $Q$ has a $k''$-point $x$ and $\varphi'$ is an ismorphism around $x$.

If $n=4$, we have
$\pic(X_{\overline k})\cong\pic(\PP^2_{\overline k}\times\PP^2_{\overline k})$ by  \cite[Corollary IV.3.3]{MR0282977} and $H$ is an element of bidegree $(1,1)$. 
By Lemma \ref{lem:splitting_products_proj_spaces} there is an extension $k'$ of $k$ of degree at most $6$ such that $X_{k'}\subseteq\PP^2_{ k'}\times\PP^2_{ k'}$  
is defined by a bihomogeneous polynomial $f$ of bidegree $(1,1)$ with coefficients in $k'$ in a set of  variables $(x_0,x_1,x_2;y_0,y_1,y_2)$. Evaluating $f$ on $y_0=y_1=y_2=1$ yields a linear form in 3 variables, which has a nontrivial zero $(\tilde x_0,\tilde x_1,\tilde x_2)\neq(0,0,0)$ over $k'$. Then $((\tilde x_0:\tilde x_1:\tilde x_2),(1:1:1))$ is a $k'$-point on $X$.

If $n=6$, we have
$\pic(X_{\overline k})\cong\pic(\PP^2_{\overline k}\times\PP^2_{\overline k}\times\PP^2_{\overline k})$ by  \cite[Corollary IV.3.3]{MR0282977} and $H$ is an element of tridegree $(1,1,1)$. 
By Lemma \ref{lem:splitting_products_proj_spaces} there is an extension $k'$ of $k$ of degree at most $18$ such that $X_{k'}\subseteq\PP^2_{k'}\times\PP^2_{ k'}\times\PP^2_{k'}$ is defined by three trihomogeneous polynomials $f_1,f_2,f_3$ of tridegrees $(0,1,1)$, $(1,0,1)$, $(1,1,0)$, respectively, with coefficients in $k'$ in a set of variables $(x_0,x_1,x_2;y_0,y_1,y_2;z_0,z_1,z_2)$. Evaluating $f_1$ and $f_2$ at $z_0=z_1=z_2=1$ gives two linear forms $\tilde f_1$, $\tilde f_2$ in two distinct sets of variables. Let $(\tilde y_0,\tilde y_1,\tilde y_2)\neq(0,0,0)$ be a nontrivial zero of $\tilde f_1$ over $k'$. Evaluating $f_3$ at $y_i=\tilde y_i$ for $i\in\{0,1,2\}$ gives a linear form $\tilde f_3$. Let  $(\tilde x_0,\tilde x_1,\tilde x_2)\neq(0,0,0)$ be a nontrivial common zero of $\tilde f_2$ and $\tilde f_3$ over $k'$. Then
$((\tilde x_0:\tilde x_1:\tilde x_2),(\tilde y_0:\tilde y_1:\tilde y_2),(1:1:1))$ is a $k'$-point on $X$.

For $n=7$, 
we conclude by Lemma \ref{lem:splitting_products_proj_spaces}.
\end{proof}

\subsubsection{Singular Gorenstein}
By \cite[Theorem 6.6]{MR3100918} singular Gorenstein terminal Fano threefolds of geometric Picard rank $>1$ and Picard rank $1$ have the same description (1.2.$n$) as in Section \ref{sec:smooth_picardrk>1} for $n\in\{1,2,3,5,6,8\}$.
\begin{lemma}
Let $n\in\{1,\dots,8\}$, and let $X$ be a singular Gorenstein terminal Fano threefold over a field $k$ of characteristic $0$ of type (1.2.$n$) as in \cite[Theorem 6.6]{MR3100918}. Then $X$ acquires a rational point over a suitable base field extension of degree at most $d_3(n)$, for
\[
(n,d_3(n))\in\{(1, 12), \ (2, 20), \ (3, 18),  \ (5, 12)  \ (6, 17),  \ (8, 17)\}.
\]
\end{lemma}
\begin{proof}
Let $\overline k$ be an algebraic closure of $k$, and let $X'$ be a smoothing of $X_{\overline k}$. Since the Picard group and the degree $(-K_{X_{\overline k}})^3$ are preserved under smoothing by \cite{MR1489117, MR2860279}, 
 $X_{\overline k}$ contains at most
$$
N(X)=20-\rk\pic(X_{\overline k}) + h^{1,2}(X')
$$
singular points by \cite[(6.3.3)]{MR3711004}.
Combining Lemma \ref{lem:basepoints} with the fact that the singular locus of $X_{\overline k}$ is defined over $k$, we conclude that $X$ acquires a rational point over suitable base field extension of degree at most $\min\{N(X),(-K_X)^3\}$. We compute $N(X)$ by consulting \cite[\S\S 12.3--12.5]{MR1668579}.
\end{proof}

\subsection{Non-Gorenstein terminal}
\begin{lemma}
Let $X$ be a non-Gorenstein terminal Fano threefold over a field $k$ of characteristic $0$. Then $X$ acquires a rational point over a suitable field extension of degree at most $16$.
\end{lemma}
\begin{proof}
Let $\overline k$ be an algebraic closure of $k$. The orbit of a non-Gorenstein $\overline k$-point on $X_{\overline k}$ under the action of $\gal(\overline k/k)$ has cardinality at most $16$ (the computation can be found in the proof of \cite[Lemma 4.2.1]{MR3711004}). 
\end{proof}

\section{Field of definition of rational points}
\label{sec:extension_points}

In this section we prove Theorem \ref{thm:extension_points} and Theorem \ref{thm:extension_points_effective}. 

\begin{proof}[Proof of Theorem \ref{thm:extension_points}]
For the first statement we proceed by induction on $n$. If $n=0$ there is nothing to prove. Let $n>0$.
Since the degree of a projective variety is encoded in the Hilbert polynomial, by Proposition \ref{prop:boundedness} there exists a positive integer $d'$ such that for every terminal Fano variety $X$ of dimension $n$ there exists an embedding $X\subseteq\PP^N_k$ of degree $\leq d'$.
By Bertini's theorem $n$ general hyperplanes of $\PP^N_k$ intersect $X$ in a smooth subvariety of dimension 0 that acquires a rational point after a field extension of degree $\leq d'$.
 We retrace the proof of Theorem \ref{thm:equivalent_formulation} to show that we can take 
 \begin{equation}\label{eq:extension_points}
 d(n)=d(n-1)\max\{d'(n), \max_{1\leq m\leq n-1}d(m)d(n-m)\},
 \end{equation}
 where $d'(n)$ is a positive integer (for example $d'$ above) such that for every field $k$ of characteristic 0, every terminal $\QQ$-factorial Fano variety of dimension $n$ and Picard rank 1 over $k$ acquires rational points over a suitable field extension of degree at most $d'(n)$ over $k$.
  Let $X$ be a proper rationally connected variety of dimension $n$. By the argument in the proof of Lemma \ref{lem:conj_sing} we can assume without loss of generality that $X$ is smooth and projective.
Let $f:X'\to Y$ be the fibration with $X'$ birationally equivalent to $X$ provided by Proposition \ref{prop:building_block}, and let $X''\to X'$ be a resolution of singularities.
Then there exists a field extension $k\subseteq k'$ of degree at most $ \max\{d'(n), \max_{1\leq m\leq n-1}d(m)d(n-m)\}$ such that $X'(k')\neq\emptyset$. Indeed,
if $\dim Y=0$, then $X'$ acquires a rational point after a suitable field extension of degree at most $ d'(n)$,  as $X'$ is a terminal $\QQ$-factorial Fano $k$-variety of dimension $n$ and Picard rank 1. If $\dim Y=m>0$, then by induction hypothesis $Y$ acquires a rational point after a suitable field extension of degree at most $d(m)$ of $k$ and the fiber of $X'\to Y$ over such a point acquires a rational point after a further suitable field extension of degree at most $ d(n-m)$ by \cite[Theorem 1.2]{MR2491906} and the induction hypothesis.
 By \cite[Theorem 1.3]{MR2491906} and the induction hypothesis there exists a finite field extension $k'\subseteq k''$ of degree $\leq d(n-1)$ such that $X''(k'')\neq\emptyset$. Then $X(k'')\neq\emptyset$ by \cite{MR0095851}.
 \end{proof}
 
 \begin{proof}[Proof of Theorem \ref{thm:extension_points_effective}]
 The bound $\tilde d(1)=2$ is immediate, because all conics acquire rational points on suitable quadratic extensions and $\{x^2+y^2+z^2=0\}$ defines a conic without rational points over $\QQ$. For $n=2$, by resolution of singularities,  \cite{MR0095851}, and the birational classification of proper smooth surfaces \cite{MR525940},  it suffices to find a bound for  del Pezzo surfaces and rational conic bundles.
 Hence, $\tilde d(2)=6$ follows from Lemmas \ref{lem:conicbundle} and  \ref{lem:delPezzo}. 
 
 For $n=3$, let $d'(3)$ be the bound for terminal Fano threefolds of Picard rank $1$. Then $d'(3)\leq 24$ by the results in Section \ref{sec:fano_3folds_existence}, and \eqref{eq:extension_points} gives   $\tilde d(3)\leq 6\max\{24,12\}=144$.
\end{proof}

\section{Varieties with points over $C_1$ fields}
\label{sec:varieties_with_points}

In this section we find rational points on a number of rationally connected varieties over $C_1$ fields. The results in this section will be used for the proof of Theorem \ref{thm:3dim_Fano_gor}.

\begin{remark}\label{rem:brauer}
We recall that $C_1$ fields have cohomological dimension $\leq1$ and hence trivial Brauer group  \cite[\S II.3.2]{MR1867431}.
In particular, projective spaces have no nontrivial forms over $C_1$ fields  \cite[\S5]{MR2266528}. 
We also recall that algebraic extensions of $C_1$ fields are $C_1$ \cite[Corollary to Theorem 5]{MR0046388}.
\end{remark}

\subsection{Toric varieties}
We prove that $k$-forms of toric varieties have rational points over $C_1$ fields $k$ of characteristic $0$.

 \begin{prop}\label{prop:toric}
 Let $k$ be a field of characteristic 0 of cohomological dimension $\leq 1$. Let $\overline k$ be a separable closure of $k$. Let $X$ be a $k$-variety such that $X_{\overline k}$ is isomorphic to a proper $\QQ$-factorial toric $\overline k$-variety. Then $X$ is an equivariant compactification of a $k$-torus and has a smooth $k$-point.
 \end{prop}
 \begin{proof}
 Let $\Sigma$ be a fan in $\ZZ^n$ and $Y$ the split toric $k$-variety associated to $\Sigma\subseteq\ZZ^n$. Assume that $Y$ is proper and $\QQ$-factorial (i.e., $\Sigma$ is simplicial and $\bigcup_{\sigma\in\Sigma}\sigma=\RR^n$).
 By \cite[Proposition III.1.3.5]{MR1867431}, the $k$-forms of $Y$ are classified by $H^1_{\text{\it\'et}}(k, \Aut_Y)$ up to isomorphism. 
 By \cite[Corollary 4.7]{MR1299003}  $\Aut_Y$ is a $k$-linear algebraic group. 
 Hence, $H^1_{\text{\it\'et}}(k, \Aut_Y)\cong H^1_{\text{\it\'et}}(k, \Aut_Y/\Aut_Y^0)$
 by \cite[Corollary III.2.4.3]{MR1867431}.
 By \cite[Corollary 4.7]{MR1299003}, there is a surjective homomorphism $\aut(\ZZ^n,\Sigma)\to \Aut_{Y}/\Aut_{Y}^0$, where $\aut(\ZZ^n,\Sigma)$ is the group of lattice automorphisms of $\ZZ^n$ that preserve the fan $\Sigma$. Therefore, there is a surjective map $H^1_{\text{\it\'et}}(k, \aut(\ZZ^n,\Sigma))\to H^1_{\text{\it\'et}}(k, \Aut_{Y}/\Aut_{Y}^0)$ by \cite[Corollary III.2.4.2]{MR1867431}.
 Moreover, to every element of $H^1_{\text{\it\'et}}(k, \aut(\ZZ^n,\Sigma))$ correspond an isomorphism class of normal $k$-varieties with a faithful action of a $k$-torus that has a dense orbit by \cite[Theorems 3.2 and 3.4]{MR3163576}. Hence, every $k$-form $X$ of $Y$ is a normal $k$-variety with a faithful action of a $k$-torus that has a dense orbit $U$. By \cite[Theorem 6]{MR651645}, $U$ is a principal homogeneous space under a torus $T$. Since tori are connected linear algebraic groups by \cite[p. 114]{MR1102012}, $H^1_{\text{\it\'et}}(k,T)=0$ by \cite[Corollary III.2.4.3]{MR1867431}. Therefore, $U\cong T$ has a $k$-point (the unit element). Since tori are smooth, $X$ has a smooth $k$-point.
 \end{proof}
Note that the assumption on the characteristic of $k$ ensures that the automorphism group of the toric variety is smooth. The author does not know whether smoothness and the results in \cite[Corollary 4.7]{MR1299003} hold in positive characteristics.
The $\QQ$-factoriality assumption can be removed using \cite[Theorem 7.8]{arXiv:1809.09070v1}.

\subsection{Intersections of low degree in products of weighted projective spaces}
\label{subsection:products_wproj}

We generalize \cite[Theorem 4]{MR0046388} and \cite[Theorem IV.6.7]{MR1440180} to orbit complete intersections (in the sense of \cite{forms_ci}) in forms of products of weighted projective spaces over $C_1$ fields of arbitrary characteristic.

We first study the forms of products of projective spaces over $C_1$ fields.

\begin{lemma}\label{lem:forms_product_proj_space}
Let $k$ be a field of cohomological dimension $\leq 1$. Let $\overline k$ be a separable closure of $k$.
Let $N$, $n_1,\dots,n_N$ and $m_1<\dots<m_N$ be positive integers. Then  
$H^1_\et\left(k,\Aut_{\prod_{i=1}^N\prod_{j=1}^{n_i}\PP^{m_i}_{k}}\right)=\prod_{i=1}^N H^1_\et\left(k,S_{n_i}\right)$.
\end{lemma}
\begin{proof}
Since automorphisms preserve the effective cone in $\pic(Y_{\overline k})$ and the intersection product, $\aut(\prod_{i=1}^N\prod_{j=1}^{n_i}\PP^{m_i}_{k})=\prod_{i=1}^N\aut(\prod_{j=1}^{n_i}\PP^{m_i}_{ k})$ and for every $i\in\{1,\dots,N\}$ there is
an exact sequence of groups
\begin{equation*}
1\to\prod_{j=1}^{n_i}\aut\left(\PP^{m_i}_{ k}\right)\to\aut\left(\prod_{j=1}^{n_i}\PP^{m_i}_{ k}\right)\to S_{n_i}\to 1.
\end{equation*}
This yields exact sequences
\begin{equation*}
\prod_{j=1}^{n_i} H^1_\et\left(k,\Aut_{\PP^{m_i}_{ k}}\right)\to H^1_\et\left(k,\Aut_{\prod_{j=1}^{n_i}\PP^{m_i}_{ k}}\right)\to H^1_\et\left(k,S_{n_i}\right).
\end{equation*}
for every $i\in\{1,\dots,N\}$, where
the second arrow is surjective by \cite[Corollary III.2.4.2]{MR1867431}, and $H^1_{\et}(k,\Aut_{\PP^{m_i}_{ k}})=0$ as $\Br(k)=0$ (cf.~Remark \ref{rem:brauer}). 
\end{proof}

\begin{notation}\label{notation:permutation_action}
For positive integers $a_1,\dots,a_m$, we denote by $\PP_k(a_1,\dots,a_m)$ the $(m-1)$-dimensional weighted projective space over $k$ with weights $a_1,\dots,a_m$.

For every $N$-tuples of positive integers $\mathbf{n}=(n_1,\dots,n_N)$ and $(m_1,\dots,m_N)$ and every $(\sum_{i=1}^Nm_i)$-tuple of positive integers $\mathbf{a}=(a_{i,l})_{\substack{1\leq i\leq N\\1\leq l\leq m_i}}$, 
we denote by $\PP(\mathbf{n},\mathbf{a}) $ a $k$-variety such that 
\[
\PP(\mathbf{n},\mathbf{a}) _{\overline k} \cong \prod_{i=1}^N\prod_{j=1}^{n_i}\PP_{\overline k}(a_{i,1},\dots,a_{i,m_i})
\]
 with a $\gal(\overline k/k)$-action by permutation of the factors with exactly $N$ orbits given by $\prod_{j=1}^{n_i}\PP_{\overline k}(a_{i,1},\dots,a_{i,m_i})_{\overline k}$ for all $i\in\{1,\dots,N\}$. That is, if we write $(x_{i,j,l})_{\substack{1\leq i\leq N\\1\leq j\leq n_i\\ 1\leq l\leq m_i}}$ for the coordinates (here  $x_{i,j,1}, \dots, x_{i,j,m_i}$ are the coordinates in the $j$-th factor $\PP_{\overline k}(a_{i,1},\dots,a_{i,m_i})$ of the $i$-th orbit), then $\gal(\overline k/k)$ acts on the coordinates by permutations of the second index in a way that $\{x_{i,1,l},\dots,x_{i,n_i,l}\}$ forms an orbit for all $i\in\{i,\dots,N\}$ and all $l\in\{1,\dots,m_i\}$.
\end{notation}

\begin{remark}\label{rem:forms_product_proj_space}
Lemma \ref{lem:forms_product_proj_space} together with \cite[Proposition III.1.3.5]{MR1867431}  shows that all forms of products of projective spaces over a field $k$ of cohomological dimension $\leq 1$ are isomorphic to $\PP(\mathbf{n},\mathbf{a})$ for some $\mathbf{n}$ and $\mathbf{a}=(1,\dots,1)$.
\end{remark}

 \begin{prop}\label{prop:i_prod_wproj}
Let $k$ be a $C_1$-field that admits normic forms of every degree. Let $\overline k$ be a separable closure of $k$. 
 Let $\mathbf{n}$ and $\mathbf{a}$ as in Notation \ref{notation:permutation_action}.
 Let $X\subseteq \PP(\mathbf{n},\mathbf{a})$ be a subvariety such that $X_{\overline k}$ is an intersection of hypersurfaces  $H_1,\dots,H_s$ in $\PP(\mathbf{n},\mathbf{a})_{\overline k}$ of weighted multidegrees $(d_{i,j,t})_{\substack{1\leq i\leq N\\1\leq j\leq n_i}}$ for $t\in\{1,\dots,s\}$, 
such that $\sum_{t=1}^sd_{i,j,t}< \sum_{l=1}^{m_i}a_{i,l}$ for all $i\in\{1,\dots,N\}$ and all $j\in\{1,\dots,n_i\}$.
Assume that for some $1\leq s_1<\dots< s_r:=s$, $s_0:=0$, the hypersurfaces $H_{s_{i-1}+1},\dots,H_{s_{i}}$ form an orbit under the $\gal(\overline k/k)$-action on $\PP(\mathbf{n},\mathbf{a})_{\overline k}$
for all $i\in\{1,\dots,r\}$.
 Then $X$ has a $k$-rational point.
\end{prop}
\begin{proof}
 Let
$
R=\overline k[x_{i,j,l}]_{\substack{1\leq i\leq N\\1\leq j\leq n_i\\1\leq l\leq m_i}}
$
be the weighted coordinate ring of $\PP(\mathbf{n},\mathbf{a})_{\overline k}$ (here the variable $x_{i,j,l}$ corresponds to the $l$-th coordinate in the $j$-th factor $\PP_{\overline k}(a_{i,1},\dots,a_{i,m_i})$ of the $i$-th orbit), then the group $\gal(\overline k/k)$ acts on the variables by permutations of the second index.
Let $f_1,\dots,f_s\in R$ that define $H_1,\dots,H_s$, respectively, and such that $\{f_{s_{i-1}+1},\dots,f_{s_i}\}$ is an orbit under the $\gal(\overline k/k)$-action on $R$ for all $i\in\{1,\dots,r\}$. 

\emph{Step 1.} Assume that $f_1,\dots,f_s$ are $\gal(\overline k/k)$-invariant and $N=1$. 
For each $l\in\{1,\dots,m_1\}$, the $\overline k$-vector space $\sum_{1\leq j\leq n_1}x_{1,j,l}\overline k$ is a $\gal(\overline k/k)$-invariant subset of $R$, hence it has a basis $z_{1,l},\dots,z_{n_1,l}$ consisting of $\gal(\overline k/k)$-invariant elements of degree $a_{1,l}$. Then a linear change of variables gives an isomorphism $\varphi:R\to\overline k[z_{j,l}]_{\substack{1\leq j\leq n_1\\1\leq l\leq m_1}}$ of $\overline k$-algebras such that, for every $t\in\{1,\dots,s\}$, $\varphi(f_t)$ is a weighted homogeneous  polynomial of weighted degree $\sum_{j=1}^{n_1}d_{1,j,t}$ in the variables $z_{1,1},\dots,z_{n_1,m_1}$.
Since $\sum_{\substack{1\leq j\leq n_1\\1\leq t\leq s}}d_{1,j,t}<n_1\sum_{l=1}^{m_1}a_{1,l}$, the system of equations $\varphi(f_1)=\dots=\varphi(f_s)=0$ has a solution $(\tilde z_{1,1},\dots,\tilde z_{n_1, m_1})\in k^{n_1m_1}\smallsetminus\{0\}$ by \cite[Theorem IV.6.7]{MR1440180}. 
Hence, the system $f_1=\dots=f_s=0$ has a $\gal(\overline k/k)$-invariant solution $(\tilde x_{1,j,l})_{\substack{1\leq j\leq n_1\\1\leq l\leq m_1}}\in\overline k^{n_1m_1}$ with $\tilde x_{1,\tilde j,\tilde l}\neq0$ for some $(\tilde j, \tilde l)\in\{1,\dots,n_1\}\times\{1,\dots,m_1\}$. 
Since $\{\tilde x_{1,1,\tilde l},\dots,\tilde x_{1,n_1,\tilde l}\}$ is an orbit under the $\gal(\overline k/k)$-action, then $\tilde x_{1,j,\tilde l}\neq0$ for all $j\in\{1,\dots,n_1\}$. 
Thus $(\tilde x_{1,j,l})_{\substack{1\leq j\leq n_1\\1\leq l\leq m_1}}$ defines a $\gal(\overline k/k)$-invariant point in $X_{\overline k}(\overline k)$, and hence a $k$-rational point on $X$.

\emph{Step 2.} Assume that $f_1,\dots,f_s$ are $\gal(\overline k/k)$-invariant and $N$ is arbitrary. We proceed by induction on $N$. The case $N=1$ is  Step 1. Assume that $N>1$. 
Let $F\subseteq\{f_1,\dots,f_s\}$ be the set of elements $f_t$ such that $d_{N,j,t}=0$ for all $j\in\{1,\dots,n_N\}$. 
We observe that $F$ is invariant under the $\gal(\overline k/k)$-action on $R$.
If $F\neq\emptyset$,
by hypothesis of induction the subvariety  $V(F)\subseteq\prod_{i=1}^{N-1}\prod_{j=1}^{n_i}\PP_{\overline k}(a_{i,1},\dots,a_{i,m_i})$ contains a $\gal(\overline k/k)$-invariant point with coordinates $(\tilde x_{i,j,l})_{\substack{1\leq i\leq N-1\\1\leq j\leq n_i\\1\leq l\leq m_i}}\in \overline k^{\sum_{i=1}^{N-1}n_im_i}$.
 If $F=\emptyset$, let $\tilde x_{i,j,l}:=1$ for all $i\in\{1,\dots,N-1\}$, $j\in\{1,\dots,n_i\}$, $l\in\{1,\dots,m_i\}$. Evaluating $f_1,\dots,f_s$ in $x_{i,j,l}=\tilde x_{i,j,l}$ for all $i\in\{1,\dots,N-1\}$, $j\in\{1,\dots,n_i\}$, $l\in\{1,\dots,m_i\}$, yields a system  of $\gal(\overline k/k)$-invariant forms $f'_1,\dots,f'_s$ in $\overline k[x_{N,j,l}]_{\substack{1\leq j\leq n_N\\1\leq l\leq m_N}}$. By Step 1 the subvariety  $V(f'_1, \dots, f'_s)\subseteq\prod_{j=1}^{n_N}\PP_{\overline k}(a_{N,1},\dots,a_{N,m_N})$ contains a $\gal(\overline k/k)$-invariant point with coordinates $(\tilde x_{N,j,l})_{\substack{1\leq j\leq n_N\\1\leq l\leq m_N}}\in \overline k^{n_Nm_N}$. Then $(\tilde x_{i,j,l})_{\substack{1\leq i\leq N\\1\leq j\leq n_i\\1\leq l\leq m_i}}$ defines a $\gal(\overline k/k)$-invariant point in $X_{\overline k}(\overline k)$, hence, a $k$-point on $X$.

\emph{Step 3.}
No restrictions on $f_1,\dots,f_s$. 
For $l=1,\dots,r$, let $h_l:=\prod_{t=s_{l-1}+1}^{s_l}f_t$. Then $h_l\in R$ is a $\gal(\overline k/k)$-invariant weighted multihomogeneous element of $R$ of weighted degree $(\sum_{t=s_{l-1}+1}^{s_l}d_{i,j,t})_{\substack{1\leq i\leq N\\1\leq j\leq n_i}}$.
The subvariety of $Y_{\overline k}$ defined by $h_1=\dots=h_r=0$ has a $\gal(\overline k/k)$-invariant $\overline k$-point $\widetilde x$ by Step 2.
The point $\widetilde x$  belongs to at least one hypersurface in each orbit under the $\gal(\overline k/k)$-action on $\{H_1,\dots,H_s\}$, hence it belongs to all. So $\widetilde x$ defines a $k$-point on $X$.
\end{proof}

\subsection{Some cyclic coverings}\label{sec:sec:cyclic_covering}
We study rational points over $C_1$ fields for  cyclic coverings of complete intersections in forms of products of projective spaces that have Picard rank 1 over the base field.
The assumption on the Picard rank is essential for the proof.

\begin{prop}\label{prop:cyclic}
Let $k$ be a $C_1$ field that admits normic forms of every degree. Let $\overline k$ be a separable closure of $k$. Let $Y$ be a $k$-variety such that $Y_{\overline k}\cong\prod_{i=1}^n\PP^{m}_{\overline k}$ and $\rk\pic(Y)=1$. Let $V\subseteq Y$ be an intersection of $r$ hypersurfaces of degrees $d_1(1,\dots,1),\dots,d_r(1,\dots,1)$ in $\pic(Y_{\overline k})$.
Let $X\to V$ be a cyclic covering of degree $d$ with reduced ramification divisor given by the restriction to $V$ of a hypersurface of $Y$ of degree $d_0(1,\dots,1)$ in $\pic(Y_{\overline k})$. 
If $\sum_{i=0}^rd_i\leq m+\frac{d_0}d$, then $X(k)\neq\emptyset$.
\end{prop}
\begin{proof}
Let $R=\overline k[x_{1,0},\dots,x_{n,m}]$ be the coordinate ring of $Y_{\overline k}$.
By Remark \ref{rem:forms_product_proj_space} we can assume that $\gal(\overline k/k)$ acts by permutations of the first index.
Since $\rk\pic(Y)=1$, the set $\{x_{1,j},\dots,x_{n,j}\}$  is an orbit under the $\gal(\overline k/k)$-action on $R$ for all $j\in\{0,\dots,m\}$.
Let $f_0,\dots,f_r\in R$ be $\gal(\overline k/k)$-invariant homogeneous elements of degrees $d_0(1,\dots,1),\dots,d_r(1,\dots,1)$ such that $f_1,\dots,f_r$ define $V$ and $f_0$ defines the ramification divisor $Z$. 

The Segre embedding
$\alpha:Y_{\overline k}\subseteq\PP^{(m+1)^n-1}_{\overline k}$ is defined over $k$. Let $S=\overline k[z_{\underline j}:\underline j\in\{0,\dots,m\}^n]$ be a coordinate ring of $\PP^{(m+1)^n-1}_{\overline k}$ such that $\alpha$ corresponds to the morphism $\varphi:S\to R$ that sends a coordinate $z_{(j_1,\dots,j_n)}$ to $\prod_{i=1}^nx_{i,j_i}$. Let $I\subseteq S$ be the kernel of $\varphi$.
 Then $S/I$ is isomorphic to a subring of $R$ that contains $f_0,\dots,f_r$. Let $\tilde f_0,\dots,\tilde f_r$ be the corresponding elements in $S/I$. 
Then $X_{\overline k}$ is the subvariety of an $(m+1)^n$-dimensional weighted projective space $\PP_{\overline k}(1,\dots,1,d_0/d)$ defined by the ideal $I+(\tilde f_1,\dots,\tilde f_r,t^d-\tilde f_0)\subseteq S[t]$ where $t$ is a new variable of degree $d_0/d$.

Let $\mathbf{n}=(n)$ and $\mathbf{a}=(\underbrace{1,\dots,1}_{m+1},d_0/d)$.
We consider 
 the $k$-variety $\PP(\mathbf{n},\mathbf{a})$
with coordinates \[((x_{1,0}:\dots:x_{1,m}:t_1),\dots,(x_{n,0}:\dots:x_{n,m}:t_n))\] 
on $\PP(\mathbf{n},\mathbf{a})_{\overline k}$
and a $\gal(\overline k/k)$-action that is compatible with the $\gal(\overline k/k)$-action on $Y_{\overline k}$ under the embedding that identifies $Y_{\overline k}$ with the subvariety of $\PP(\mathbf{n},\mathbf{a})_{\overline k}$ defined by $t_1=\dots=t_n=0$.
The coordinate ring of 
$\PP(\mathbf{n},\mathbf{a})_{\overline k}$
is $R':=R[t_1,\dots,t_n]$, where $t_1,\dots,t_n$ are new variables with $\deg t_i=\frac{d_0}d\deg x_{i,0}$ for all $i\in\{1,\dots,n\}$.
We observe that $f_1,\dots,f_r,(t_1\cdots t_n)^{d}-f_0$ define a $\gal(\overline k/k)$-invariant subvariety of $\PP(\mathbf{n},\mathbf{a})_{\overline k}$, which
has a  $\gal(\overline k/k)$-invariant $\overline k$-point $\tilde x=(\tilde x_{1,0}:\dots:\tilde x_{n,m}:\tilde t_1:\dots:\tilde t_n)$ by Proposition \ref{prop:i_prod_wproj}. 
If $\tilde x_{i,j}\neq0$ for some $(i,j)\in\{1,\dots,n\}\times\{0,\dots,m\}$, then $\tilde x_{1,j},\dots,\tilde x_{n,j}\neq0$ as they form an orbit under the $\gal(\overline k/k)$-action on $\overline k$.
If $\tilde x_{1,0},\dots,\tilde x_{n,m}=0$, then $t_1,\dots,t_n\neq0$ as $\tilde x$ is a point of $\prod_{j=1}^n\PP^{m+1}_{\overline k}$. Let  $\tilde t:=\tilde t_1\cdots\tilde t_n$, and let $\tilde z_{\underline j}$ be the evaluation of $\varphi(z_{\underline j})$ at $(\tilde x_{1,0},\dots,\tilde x_{n,m})$ for all $\underline j\in\{0,\dots,m\}^n$. The $\overline k$-point of $\PP(1,\dots,1,d_0/d)$ with coordinates 
$(\{\tilde z_{\underline j}\}_{\underline j\in \{0,\dots,m\}^n},\tilde t)$ is then a $\gal(\overline k/k)$-point of $X_{\overline k}$.
\end{proof}

\subsection{A symmetric Cremona transformation}
\label{sec:cremona}

In this section we consider the varieties described in \cite[(1.2.3)]{MR3100918} (cf.~Section \ref{sec:rkpic>1}).
Some of the computations have been carried out using the web interface SageMathCell of the  software  \cite{sage}. 
%\extra{See documents ActionCremonaSage and SmoothCremonaSage}
\begin{notation}
We denote the adjugate of an invertible matrix $M$ by $Ad(M):=\det(M) M^{-1}$.
We recall that, given an automorphism $\sigma$ of a field $k$, a morphism $\theta:V\to W$ of $k$-vector spaces  is called $\sigma$-linear if $\theta(av)=\sigma(a)\theta(v)$ for all $a\in k$ and $v\in V$.
We denote by $\sigma:\PP^n(k)\to\PP^n(k)$ the $\sigma$-linear automorphism that sends $(y_0:\dots:y_n)$ to $(\sigma^{-1}(y_0):\dots:\sigma^{-1}(y_n))$.
\end{notation}

Let $k$ be a field of characteristic 0.
Let $S\subseteq\PP^5_{k}$ be the image of the Veronese embedding $v:\PP^2_{k}\to\PP^5_{k}$,
\[
v(x_0:x_1:x_2)=(x_0^2:x_1^2:x_2^2:x_1x_2:x_0x_2:x_0x_1).
\]
Denote by 
$(y_0:\dots:y_5)$ the coordinates on $\PP^5_k$.
Then $S$ is defined by the quadrics 
\begin{gather*}
A_{0}:=y_{1}y_{2}-y_{3}^2,\quad A_{1}:=y_{0}y_{2}-y_{4}^2,\quad A_{2}:=y_{0}y_{1}-y_{5}^2,\\A_{3}:=y_{4}y_{5}-y_{0}y_{3},\quad A_{4}:=y_{3}y_{5}-y_{1}y_{4},\quad A_{5}:=y_{3}y_{4}-y_{2}y_{5},
\end{gather*} 
and $A_0,\dots,A_5$ is a basis of the linear system $\mathscr{S}$ of quadrics in $\PP^5_k$ containing $S$.

Let $S_3$ be the $k$-vector space of  symmetric $3\times 3$ matrices. We consider the isomorphism $\PP^5_k\to\PP(S_3)$ induced by the isomorphism
\begin{gather*}
\alpha:k^6\to S_3,\quad (y_0,\dots,y_5)\mapsto\begin{pmatrix}y_0&y_5&y_4\\y_5&y_1&y_3\\y_4&y_3&y_2\end{pmatrix}.
\end{gather*}
Every automorphism of $\PP^5_k$ that preserves $S$ belongs to the image of the embedding $\PGL(3,k)\subseteq \PGL(6,k)$ induced by the group homomorphism $\varphi:\GL(3,k)\subseteq \GL(6,k)$ that sends a matrix $M\in \GL(3,k)$ to the element of $\GL(6,k)$ defined by $\varphi(M)(y)=\alpha^{-1}(M\alpha(y)M^{t})$ for all $y\in k^6$. 
%\extra{see Sage document ActionCremonaSage part 1, 2}

Let $A:\PP^5_{k}\dashrightarrow\PP^5_{k}$ be the rational map defined by $(A_0,\dots,A_5)$. It is a symmetric Cremona transformation.
For every matrix $M\in \GL(3,k)$ the birational map $\varphi(M)\circ A$ is a Cremona transformation of $\PP^5_k$ with inverse $\varphi(M^t)\circ A$.
%\extra{see Sage document ActionCremonaSage part 3}

Let $f:Y\to\PP^5_k$ be the blow up with center $S$. Then $f$ resolves any Cremona transformation $\PP^5_k\dashrightarrow\PP^5_k$ induced by the linear system $\mathscr{S}$  
(see \cite{MR1020829} for example). 
 Let $H$ be the pullback of a hyperplane in $Y$ and $E$ the exceptional divisor of $f$. Then $|2H-E|$ is the linear system of strict transforms of elements of $\mathscr{S}$ under $f$. 

\begin{lemma}\label{lem:galoisaction_cremona}
Let $L$ be a hyperplane in $\PP^5_k$ such that $L\cap S$ spans $L$. Let $Q$ be a quadric in $\mathscr{S}$.
Let $X\subseteq Y$ be the strict transform of $L\cap Q$.
Let $\sigma$ be an automorphism of $k$.
Then every $\sigma$-linear automorphism $\theta$ of $X$ that induces a $\sigma$-linear isomorphism between the complete linear systems of $H|_X$ and $(2H-E)|_X$ is the restriction to $X$ of a $\sigma$-linear automorphism $\theta_Y$ of $Y$ that resolves a $\sigma$-linear Cremona transformation on $\PP^5_k$ of the form $\sigma\circ\varphi(M)\circ A$ for some $M\in \GL(3,k)$. Moreover, if $\theta$ is an involution, so is $\theta_Y$.
\end{lemma}
\begin{proof}
We fix a basis $z_0,\dots,z_5$ of the linear system $|H|$ such that $z_5$ defines the hyperplane $L$. Let $C\in \GL(6,k)$ such that the basis $y_0,\dots,y_5$ of $|H|$
corresponding to the choice of coordinates  on $\PP^5_k$ can be written as $(y_0,\dots,y_5)=C(z_0,\dots,z_5)$.
 Let $B_0,\dots,B_4, \tilde B_0, \dots,\tilde B_4\in \mathscr{S}$ such that $B_0|_{Q},\dots,B_4|_{Q}$ and $\tilde B_0|_{Q},\dots,\tilde B_4|_{Q}$ are the bases of $\mathscr{S}|_{Q}$ image of the basis $z_0|_L,\dots,z_4|_L$ of $|H||_L$ under the isomorphisms induced by $\theta$ and $\theta^{-1}$, respectively.
Let $R, \tilde R\in \GL(6,k)$ such that 
\[\sigma(C)(B_0,\dots,B_4, Q)=R (A_0,\dots,A_5), \quad \sigma(C)(\tilde B_0,\dots,\tilde B_4, Q)=\tilde R (A_0,\dots,A_5)\]
as bases of $\mathscr{S}$. 
Then $\theta$ induces via $f$ a birational map $(\sigma\circ R\circ A)|_{Q\cap L}:Q\cap L\to Q\cap L$ with inverse $(\sigma\circ\tilde R\circ A)|_{Q\cap L}$. Since $\sigma(S)=S$ by construction and $(\sigma\circ\tilde R\circ A)|_{Q\cap L}=(A\circ R^{-1}\circ\sigma^{-1})|_{Q\cap L}$ is a birational map with exceptional locus $S\cap L=\sigma(R(S\cap L))$, we have $R(S\cap L)=S\cap\sigma^{-1}(L)$; that is, $R$ induces an isomorphism of conics between $v^{-1}(S\cap L)$ and $ v^{-1}(S\cap\sigma^{-1}(L))$. Hence there is a matrix $M\in \GL(3,k)$ such that $(v^{-1}\circ R\circ v)|_{v^{-1}(S\cap L)}$ is the restriction of the automorphism of $\PP^2_k$ defined by $M$.
 Then $\varphi(M^{-1})|_{\sigma^{-1}(L)}=R^{-1}|_{\sigma^{-1}(L)}$ because they agree  on $S\cap \sigma^{-1}(L)$ and $S\cap\sigma^{-1}(L)$ spans $\sigma^{-1}(L)$.
So $(A\circ R^{-1}\circ\sigma^{-1})|_{Q\cap L}=(A\circ\varphi(M^{-1})\circ\sigma^{-1})|_{Q\cap L}$, and the $\sigma$-linear Cremona transformations on $\PP^5_k$ induced by $R$ and $\varphi(M)$ have the same restriction to $Q\cap L$.

If $\theta\circ\theta$ is the identity on $X$, then the automorphism
\[
\varphi(\sigma(M)\cdot Ad(M)^t)=\varphi(\sigma(M))\circ A\circ\varphi(M)\circ A
\] of $\PP^5_k$ restricts to the identity on an open subset of $Q\cap L$, and hence on $L$.
 Then $\sigma(M)\cdot Ad(M)^t$ is an automorphism of $\PP^2_k$ that restricts to the identity on $v^{-1}(L\cap S)$. Since the conic $v^{-1}(L\cap S)$ spans $\PP^2_k$, we conclude that $\sigma(M)\cdot Ad(M)^t$ is the identity on $\PP^2_k$, and so $\sigma\circ\varphi(M)\circ A$ is a symmetric $\sigma$-linear Cremona transformation.
\end{proof}

Let \[\beta:k^6\to k^6, \quad (a_0,\dots,a_5)\mapsto (a_0,a_1,a_2,2a_3,2a_4,2a_5).\]
Let $\psi:\GL(3,k)\to \GL(6,k)$ be the group homomorphism defined by $\psi(M)=\beta\circ\varphi(Ad(M))\circ \beta^{-1}$. Let $M\in \GL(3,k)$ and $a=(a_0,\dots,a_5)\in k^6$, then \begin{equation}\label{eq:sumAvarphiM}
\sum_{i=0}^5a_iA_i(\varphi(M)(y))=\sum_{i=0}^5\psi_i(M)(a)A_i(y)\end{equation}
 for all  $y\in k^6$, where $(\psi_0(M)(a),\dots,\psi_5(M)(a))=\psi(M)(a)=\varphi(Ad(M)^t)^t(a)$.
 In particular, 
\begin{equation}\label{eq:AvarphiMA}
A\circ \varphi(M)\circ A=\varphi(Ad(M)^t)
\end{equation} as birational maps on $\PP^5_k$.
%\extra{see Sage document ActionCremonaSage part 4, 5}

The secant variety $V$ of $S$  in $\PP^5_k$ is the cubic hypersurface defined by 
\[
y_0y_1y_2+2y_3y_4y_5-y_0y_3^2-y_1y_4^2-y_2y_5^2=0.
\]
The action of $\PGL(3,k)$ on $\PP^5_k$ induced by $\varphi$ has three orbits: $S$, $V\smallsetminus S$
and $\PP^5_k\smallsetminus V$, 
see \cite[\S 3.1]{MR3795470} for example.
 Hence, the action of $\PGL(3,k)$ on $\PP^5_k$ induced by $\psi$ has three orbits: $\widetilde \beta(S)$, $\widetilde\beta(V)$ and $\PP^5_k\smallsetminus\widetilde\beta(V)$, where $\widetilde\beta$ is the automorphism of $\PP^5_k$ induced by $\beta$.
Now, $\widetilde\beta(S)$ is defined by $A_0\circ\beta^{-1},\dots,A_5\circ\beta^{-1}$,
and $\widetilde\beta(V)$ is defined by $p=0$, where
\[p=4y_0y_1y_2+y_3y_4y_5-y_0y_3^2-y_1y_4^2-y_2y_5^2.
\]

\begin{remark}\label{rem:smooth_quadric_cremona}
If $a=(a_0,\dots,a_5)\in k^6$ and  $q_a:=\sum_{i=0}^5a_iA_i$, the matrix $J_a$ that defines the linear system 
\begin{equation}\label{eq:cremonajacobian}
\frac{\partial q_a}{\partial y_0}=\dots=\frac{\partial q_a}{\partial y_5}=0,\end{equation}
 satisfies $\det J_a=-p(a)^2$ 
%\extra{ see Sage document SmoothCremonaSage part 1} 
 and every $5\times 5$ minor of $J_a$ is divisible by $p(a)$. So, if nonempty, the linear subvariety of $\PP^5_{k}$ defined by \eqref{eq:cremonajacobian} has dimension $\geq1$, and hence it intersects the quadric $q_a=0$. In particular, $q_a$ is smooth if and only if $p(a)\neq0$.
\end{remark}

\subsubsection{Smooth case}
\begin{lemma}\label{lem:2_3_polynomialintersection}
Let $k$ be an algebraically closed field. For $a_0,\dots,a_3,b_0,b_1,b_2\in k$, let
$
g_3:=a_3x^3+a_2x^2+a_1x+a_0$, $g_2:=b_2x^2+b_1x+b_0 \in k[x],
$ 
and
\[
\delta(g_3,g_2):=b_2\delta_1^2-b_1\delta_1\delta_2+b_0\delta_2^2,
\]
where $\delta_{1}:=a_3b_0b_1-a_2b_0b_2+a_0b_2^2$ and 
$\delta_{2}:=a_3b_1^2-a_3b_0b_2-a_2b_1b_2+a_1b_2^2$.
If  $\delta(g_3,g_2)\neq0$, then $b_2\neq0$, and $g_3$ and $g_2$ have no common zero in $k$.
\end{lemma}
\begin{proof}
If $b_2=0$, then $\delta(g_3,g_2)=0$. Assume now that $b_2\neq0$ and that $g_3$ and $g_2$ have a common zero $c\in k$.
Then $c^2=-\frac{b_1}{b_2} c-\frac{b_0}{b_2}$ and substituting it twice in $g_3$ gives
$\delta_2 c+\delta_1=0$. If $\delta_2=0$, then $\delta_1=0$. If $\delta_2\neq0$, we substitute $c=-\delta_1/\delta_2$ in $g_2$ and obtain $\delta(g_3,g_2)=0$.
\end{proof}

\begin{lemma}\label{lem:quadricintersection_cremona}
Let $k$ be an algebraically closed field of characteristic $0$. 
Let $Q$ be a smooth quadric in the linear system $\mathscr{S}$. Let $H$ be a hyperplane in $\PP^5_k$. Then for a general quadric $Q'$ in $\mathscr{S}$, the singular locus $\sing(Q\cap Q')$ is contained in the union of three lines of $\PP^5_k$, $\sing(Q\cap Q')\cap S$ is nonempty of dimension $0$, and $\sing(Q\cap Q')\cap S\cap H=\emptyset$.
\end{lemma}
\begin{proof}
By Remark \ref{rem:smooth_quadric_cremona}, the open orbit under the action of $\PGL(3,k)$ on $\PP^5(k)$ via $\psi$ is the locus of $(a_0:\dots:a_5)\in\PP^5(k)$ such that  the quadric  $\sum_{i=0}^5a_iA_i=0$ is smooth.
Hence, by \eqref{eq:sumAvarphiM}, up to an automorphism of $\PP^5_k$ that fixes $S$, we can assume that $Q$ is defined by the quadratic form $q:=A_0+A_3$.
Let $a:=(1,0,0,1,0,0)$. 
Let $q'(u):=\sum_{i=0}^5u_iA_i\in k[u_0,\dots,u_5, y_0,\dots,y_5]$, where $u:=(u_0,\dots,u_5)$.
For every $b\in k^6$, let $Q'_b\subseteq\PP^5_{k}$ be the quadric defined by the quadratic form $q'_b:=q'(b_0,\dots,b_5)\in k[y_0,\dots,y_5]$.
 
Let $L\subseteq\GG_{m,k}\times\PP^5_k\times\PP^5_k$ be the subvariety defined by  
 \begin{equation}\label{eq:lem:quadricintersection_cremona_partial}
 \frac{\partial (q-tq'(u))}{\partial y_0}=\dots=\frac{\partial (q-tq'(u))}{\partial y_5}=0,
 \end{equation}
% \extra{see Sage document SmoothCremonaSage part 7}
 where $(t,(u_0:\dots:u_5),(y_0:\dots:y_5))$ denotes the coordinates on $\GG_{m,k}\times\PP^5_k\times\PP^5_k$.
 Let $J(t,u)\in Mat_{6\times 6}(k[t,u_0,\dots,u_5])$ be the matrix that defines  \eqref{eq:lem:quadricintersection_cremona_partial} as a linear system in the variables $y_0,\dots,y_5$ with coefficients in $k[t,u_0,\dots,u_5]$.
 Then $\det(J(t,u))=-p(a-tu)^2$ and every $5\times 5$ minor of $J_{t,u}$ is divisible by $p(a-tu)$.
 
  For every $b\in k^6$,
 the singular locus of $Q\cap Q'_b$ is contained in the subvariety of $\PP^5_k$ where the Jacobian matrix of $(q,q'_b)$ has rank 1; that is, 
 the union of the linear subspaces $L_{c,b}$, for  $c\in k^\times$, where $L_{c,b}\subseteq\PP^5_k$ is the subvariety defined by the linear system $J(c,b)(y_0,\dots,y_5)=0$.
We observe that $L_{c,b}\neq\emptyset$ if and only if $p(a-cb)=0$ and in that case $\dim L_{c,b}\geq1$.
Moreover, $\dim L_{c,b}\geq2$ if and only if all the $4\times4$ minors of $J(c,b)$ vanish. 

The $4\times4$ minor $m(t,u)$ of $J(t,u)$ obtained by deleting the 4th and 6th rows and the  4th and 6th columns is
$-4u_1^2t^3(A_1(\beta^{-1}(u))t-u_2)$.
%\extra{see Sage document SmoothCremonaSage parts 2, 8}
 If $b\in k^6$ satisfies
 \begin{equation}\label{eq:B1(w)not0}
 b_1A_1(\beta^{-1}(b))p(A_1(\beta^{-1}(b))a+b_2b)\neq0,\end{equation}
then $m(t,b)$ and $p(a-tb)$ have no common zeros in $k^\times$.
Hence, for general $b\in k^6$, the singular locus of $Q\cap Q'_b$ is contained in a union of three lines of $\PP^5_k$, as $p(a-tb)\in k[t]$ has degree $\leq3$.

The variety $L$ is defined by  equations 
\begin{equation}\label{eq:equationsL(t,b)}
y_i=l_{t,u,i}(y_3,y_5), \qquad i\in\{0,1,2,4\},\end{equation} 
where $l_{t,u,i}$ are linear forms in $k(t,u_0,\dots,u_5)[y_3,y_5]$ for $i\in\{0,1,2,4\}$. 
If $b\in k^6$ satisfies $b_2p(b_0a-b)\neq0$ and \eqref{eq:B1(w)not0}, then for every $c\in k^\times$ such that $L_{c,b}\neq\emptyset$, the line  $L_{c,b}$ is defined by the linear equations $y_i=l_{c,b,i}(y_3,y_5)$, for $i\in\{0,1,2,4\}$, obtained by evaluating \eqref{eq:equationsL(t,b)} in $t=c$ and $u=b$. 
%\extra{see Sage document SmoothCremonaSage part 3}
 By substituting these equations  into the equations that define $S$, 
we see that $L_{c,b}\cap S$ is the subvariety of $L_{c,b}$ defined by the quadratic equation
\begin{equation}\label{eq:lem:quadricintersection_cremona_S}
(1-cb_0)y_3^2+cb_4y_3y_5-cb_2y_5^2=0.
\end{equation} 
%\extra{see Sage document SmoothCremonaSage part 4}
In particular, it is nonempty and of dimension 0.

Let $h_0,\dots,h_5\in k$ such that $h:=\sum_{i=0}^5h_iy_i$ is a linear form that defines $H$.
Let $H_3(t,u), H_5(t,u)\in k[t,u_0,\dots,u_5]$, such that 
 $H_3(t,u)y_3+H_5(t,u)y_5\in k(t,u_0,\dots,u_5)[y_3,y_5]$ is the  polynomial obtained by substituting \eqref{eq:equationsL(t,b)} into $m(t,u)h$.
%\extra{see Sage document SmoothCremonaSage part 5}
If $b\in k^6$ satisfies \eqref{eq:B1(w)not0}, then  $m(c,b)\neq0$ for all $c\in k^\times$ such that $L_{c,b}\neq\emptyset$, so
the intersection $L_{c,b}\cap H$ is the subvariety of $L_{c,b}$ defined by the equation $H_3(c,b)y_3+H_5(c,b)y_5=0$.

We use Lemma \ref{lem:2_3_polynomialintersection} to show that for general $b\in k^6$ the polynomials $H_3(t,b)$, $H_5(t,b)$ and $p(a-tb)$ have no common solutions, that is, there is no $c\in k^\times$ such that $L_{c,b}\neq\emptyset$ and $L_{c,b}\subseteq H$. Indeed, we observe that $H_5(t,u)/t\in k[t,u_0,\dots,u_5]$ is a polynomial of degree 2 in $t$, and that $\delta(p(a-tu),H_5(t,u)/t)$ is a nonzero element in $k[u_0,\dots,u_5]$ if $h_0\neq0$ or $h_1\neq0$ or $h_2\neq0$ or $h_4\neq0$ or $h_5\neq0$.
%\extra{see Sage document SmoothCremonaSage part 9}
 If $h_0=h_1=h_2=h_4=h_5=0$, we observe that $H_3(t,u)/t\in k[t,u_0,\dots,u_5]$ is a polynomial of degree 2 in $t$ and $\delta(p(a-tu),H_3(t,u)/t)$ is a nonzero element in $k[u_0,\dots,u_5]$,  as $h_3\neq0$.
%\extra{see Sage document SmoothCremonaSage part 10}

Then, for general $b\in k^6$ and for every $c\in k^\times$ such that $L_{c,b}\neq\emptyset$,  the intersection $L_{c,b}\cap H$ consists of a point $z$, and
 $L_{c,b}\cap H\cap S$  is defined by evaluating \eqref{eq:lem:quadricintersection_cremona_S} in $z$. 
%\extra{see Sage document SmoothCremonaSage part 6}
Computations show that since $p(a-cb)=0$, if $m(c,b)\neq0$ then $L_{c,b}\cap H\cap S$ is nonempty if and only if $C(c,b)=0$, where
\begin{multline}
C(t,u):=w_0^2h_0^2+w_1^2h_1^2+w_2^2h_2^2+w_1w_2h_3^2+w_0w_2h_4^2+w_0w_1h_5^2
+(w_5^2-2w_0w_1)h_0h_1\\+(w_4^2-2w_0w_2)h_0h_2
-(w_0w_3-w_4w_5)h_0h_3+w_0w_4h_0h_4+w_0w_5h_0h_5\\
+(w_3^2-2w_1w_2)h_1h_2+w_1w_3h_1h_3
-(w_1w_4-w_3w_5)h_1h_4+w_1w_5h_1h_5
+w_2w_3h_2h_3\\+w_2w_4h_2h_4-(w_2w_5-w_3w_4)h_2h_5
+w_2w_5h_3h_4+w_1w_4h_3h_5+w_0w_3h_4h_5
\end{multline}
%\extra{see Sage document SmoothCremonaSage part 11}
with $w_i:=a_i-tu_i\in k[t,u_0,\dots,u_5]$ for all $i\in\{0,\dots,5\}$.
Since $\delta(p(a-tu),C(t,u))$ is a nonzero element in $k[u_0,\dots,u_5]$, we conclude by Lemma \ref{lem:2_3_polynomialintersection} that $L_{c,b}\cap H\cap S=\emptyset$ for general $b\in k^6$ and for every $c\in k^\times$ such that $L_{c,b}\neq\emptyset$.
%\extra{see Sage document SmoothCremonaSage part 12}
\end{proof}

\begin{prop}\label{prop:(1.2.3)}
Let $X$ be a variety \cite[(1.2.3)]{MR3100918} defined over a field $k$ of characteristic $0$ and cohomological dimension $\leq 1$. Then $X(k)\neq\emptyset$.
\end{prop}
\begin{proof}
By \cite[pp.426--427]{MR3100918}, there exists a quadratic extension $k'$ of $k$ such that $X_{k'}$ is a blowing up of a smooth quadric $Q'\subseteq\PP^4_{k'}$ along a twisted quartic curve $\Gamma$. 
 Let $\pi:X_{k'}\to\PP^4_{k'}$ be the induced morphism. 
We identify $\PP^4_{k'}$ with a hyperplane $L_1$ of $\PP^5_{k'}$ such that $L_1\cap S=\Gamma$. Let $Q_1\subseteq\PP^5_{k'}$ be the unique quadric containing $S$ such that $Q_1\cap L_1=Q'$. Then $X_{k'}$ embeds into $Y_{k'}$ as the strict transform of $Q_1\cap L_1$ under $f$, and $\pi=f|_{X_{k'}}$.
 By Lemma \ref{lem:galoisaction_cremona} the action of $\gal(k'/k)$ on $X_{k'}$ is induced by an action of $\gal(k'/k)$ on $Y_{k'}$.
Let $\sigma$ be a generator of $\gal(k'/k)$. Then $f(\sigma(f^{-1}(L_1)))=Q_1$. 

By construction, $v^{-1}(\Gamma)$ is a smooth conic, so $\Gamma(k')$ is nonempty and hence dense in $\Gamma$ by Remark \ref{rem:brauer}. 
 Then the set of hyperplanes of $\PP^5_{k'}$ that intersect $\Gamma(k')$ is dense in $\PP(H^0(\PP^5_{k'},\OO_{\PP^5_{k'}}(1)))$.
  So we can choose a general hyperplane $L_2\subseteq\PP^5_{k'}$ that intersects $\Gamma(k')$,  such that $Q_2:=f(\sigma(f^{-1}(L_2))$ is a general member of $\mathscr{S}$, and hence $Q_1\cap Q_2$ is smooth outside the base locus $S$ of $\mathscr{S}$ by Bertini \cite[Theorem (4.1)]{MR1661859} and smooth at all $k'$-points in  $L_1\cap S$ by Lemma \ref{lem:quadricintersection_cremona}.
 
Let $x\in L_2\cap \Gamma(k')$.
Let $Z$ be the strict transform of $Q_1\cap Q_2$ under $f$. Then the morphism $f|_{Z}:Z\to Q_1\cap Q_2$, which is the blowing up of $Q_1\cap Q_2$ at $S$, is an isomorphism around $x$ because $S$ is a smooth divisor in $Q_1\cap Q_2$  and $x$ is a smooth point of $Q_1\cap Q_2$. We observe that  $f(\sigma(Z))=L_1\cap L_2$. 
 Let $L_3,L_4,L_5\subseteq\PP^5_{k'}$ be hyperplanes such that $L_1\cap L_2\cap L_3\cap L_4\cap L_5=\{x\}$. 
 Then 
\[f|_Z^{-1}(x)=Z\cap \sigma(Z)\cap \bigcap_{i=3}^5(f^{-1}(L_i)\cap \sigma(f^{-1}(L_i)))
\] consists of a $\gal(k'/k)$-invariant point of $X_{k'}$. 
\end{proof}

\subsubsection{Singular case}
\begin{prop}\label{prop:(1.2.3)sing}
Let $X$ be a variety \cite[Theorem 6.6(iii)]{MR3100918} defined over a field $k$ of characteristic $0$ with $\Br(k)=0$. Then $X$ has a $k$-rational point. 
\end{prop}
\begin{proof}
By \cite[p.432]{MR3100918} there exists a quadratic extension $k'$ of $k$, contained in an algebraic closure $\overline k$ of $k$, such that $X_{k'}$ is a blowing up of a quadric cone $Q'\subseteq\PP^4_{k'}$ with center a curve $\Gamma$ such that $\Gamma$ does not contain the vertex $\nu$ of $Q'$ and $\Gamma_{\overline k}$ is the union of two conics $\Gamma^{(1)}$ and $\Gamma^{(2)}$ that intersect each other transversally. 
Let $\sigma$ be the generator of $\gal(k'/k)$, and let $\sigma_X$ be the induced $\sigma$-linear automorphism of $X_{k'}$. 
By \cite[p.432]{MR3100918} $X$ contains at least 2 singular points.
 Since $f(\sing(X_{\overline k}))\smallsetminus\Gamma=\{\nu\}$, the $k'$-point $\gamma:=f(\sigma_X(f_1^{-1}(\nu)))$ belongs to $\sing(\Gamma)=\Gamma^{(1)}\cap\Gamma^{(2)}$.
Let $\ell$ be the strict transform of the line between $\nu$ and $\gamma$. 
Since $\ell\cong\PP^1_{k'}$, 
 it suffices to show that $\ell$ is invariant under $\sigma_X$, because in that case $\ell$ contains a $\sigma_X$-invariant $k'$-point, as $\Br(k)=0$.

Let $\pi: X_{k'}\to\PP^4_{k'}$ be the morphism induced by the blowing up.
Let $H'$ be the pullback of a hyperplane of $\PP^4_{k'}$ under $\pi$. Let $E'$ be the exceptional divisor of the blowing up. 
Then $\sigma_X(H')=2H'-E'$ in $\pic(X_{k'})$ by \cite[(5.2.1), Theorem 6.1, and p.432]{MR3100918}.
Let $\mathcal{S}$ be the linear system of quadrics in $\PP^4_{k'}$ that contain $\Gamma$.
Since $\pi$ is defined by the complete linear system $|H'|$, 
 $|2H'-E'|$ is the linear system of strict transforms of elements of $\mathcal{S}|_{Q'}$. Thus $\dim\mathcal{S}-1=\dim\mathcal{S}|_{Q'}=\dim |2H'-E'|=\dim H'=5$ as $k'$-vector spaces. 

Let $P_i\subseteq\PP^4_{\overline k}$ be the plane spanned by $\Gamma^{(i)}$ for $i\in\{1,2\}$. 
We claim that $P_1\cap P_2=\{\gamma\}$. Clearly $P_1\cap P_2$ contains $\gamma$ because it contains the intersection $\Gamma^{(1)}\cap \Gamma^{(2)}$.
If $P_1=P_2$ then $\dim \mathcal{S}\geq 9$ as $k'$-vector space, which is a contradiction.
 If $P_1\cap P_2$ has dimension 1, the linear system $\mathcal{S}'$ of quadrics containing $P_1\cup P_2$ has dimension 6, hence $\mathcal{S}=\mathcal{S}'$. So the base locus of $\mathcal{S}$ contains $(P_1\cup P_2)\smallsetminus\Gamma$, contradicting the fact that $\sigma_X(|H'|)=|2H'-E'|$ is base point free.

Now we show that $X$ is a degeneration of a Fano threefold as in Proposition \ref{prop:(1.2.3)}.
Let $k''/k'$ be a quadratic extension such that $\Gamma^{(1)}$ and $\Gamma^{(2)}$ are defined on $k''$. 
Let $u\in k''$ such that $k''=k'(u)$ and $u^2\in k'$.
 Without loss of generality, we can choose coordinates $z_0,\dots,z_4$ on $\PP^4_{k'}$ such that $\Gamma^{(1)}=\{z_3-uz_2=z_1-uz_4=z_0z_2-z_4^2\}$ and $\Gamma^{(2)}=\{z_3+uz_2=z_1+uz_4=z_0z_2-z_4^2\}$.
 Let $w:\PP^4_{k'}\subseteq\PP^5_{k'}$ be the embedding given by $(z_0:\dots:z_4)\mapsto(z_0:u^2z_2:z_2:z_3:z_4:z_1)$. Then the image of $w$ is the hyperplane $L=\{y_1-u^2y_2=0\}$, also $w(\Gamma)=L\cap S$ and $v^{-1}(L\cap S)$ is the degenerate conic  $\{x_1^2-u^2x_2^2=0\}\subseteq\PP^2_{k'}$.
Since $\dim\mathcal{S}=6=\dim\mathscr{S}$ as $k'$-vector spaces and  $\mathscr{S}|_L=\mathcal{S}$, 
there exists a unique quadric $Q$ in $\mathscr{S}$ such that $Q\cap L=w(Q')$. 
Then $X_{k'}$ embeds into $Y_{k'}$ as the strict transform of $Q\cap L$ under $f$, and $\pi=f|_{X_{k'}}$. 

By Lemma \ref{lem:galoisaction_cremona} there is $M\in \GL(3,k')$ such that the $\sigma$-linear Cremona transformation
$\sigma\circ\varphi(M)\circ A$ defines a $\sigma$-linear involution of $Y$ that restricts to $\sigma_X$ on $X_{k'}$. 
The quadric $Q$ is the image of $L$ under  $\sigma\circ\varphi(M)\circ A$ and has equation $q=(0,1,-\sigma(u)^2,0,0,0)\cdot \varphi(M)\cdot (A_0,\dots,A_5))$.
 Let $U:=(Q\cap L)\smallsetminus\pi(E'\cup\sigma_X(E'))$. Computations show that $\pi(\ell)\cap\pi(E'\cup\sigma_X(E'))=\{\gamma,\nu\}$, and $(\sigma\circ\varphi(M)\circ A)(\pi(\ell)\cap U)=\pi(\ell)\cap U$. Thus, $\ell$ is invariant under $\sigma_X$.
%\extra{see  Sage document ActionCremonaSage part 7}
  \end{proof}

\section{Terminal Fano varieties of large index}
\label{sec:fano_large_index}

 \begin{proof}[Proof of Theorem \ref{thm:largeindex_Fano}]
We denote by $\overline k$ an algebraic closure of $k$.
We recall that  $r(X)\leq n+1$ by a result of Shokurov (cf.~\cite[Corollary 2.1.13]{MR1668579}). 

  Gorenstein terminal Fano varieties of index $r(X)\geq n$ are classified in \cite[Theorem 0]{MR664549} (cf.~\cite[Theorem 3.1.14]{MR1668579}).
 If $r(X)=n+1$, then $X_{\overline k}\cong\PP^n_{\overline k}$, hence, $X\cong\PP^n_k$  by Remark \ref{rem:brauer}.
If $r(X)=n$, then the linear system $|-\frac 1nK_X|$ embeds $X$ as a quadric hypersurface in $\PP^{n+1}_k$. Hence $X(k)\neq\emptyset$ by definition of $C_1$ field.
 
Gorenstein terminal Fano varieties of index $r(X)=n-1$ are del Pezzo varieties (see \cite[Remarks 3.2.2]{MR1668579}). 
Let $X$ be a del Pezzo $k$-variety of dimension $n$. We prove that $X(k)\neq\emptyset$ by induction on $n$.
If $n=2$, $X_{\overline k}$ is a smooth rational surface, hence $X(k)\neq\emptyset$ by \cite[Proposition 2]{MR934267}.
Assume that $n>2$. By \cite[Proposition 3.2.3]{MR1668579}
the general element of the linear system $-\frac1{n-1}K_X$ is a del Pezzo $k$-variety of dimension $n-1$, hence it has a $k$-rational point by inductive hypothesis. Thus $X(k)\neq\emptyset$.
 
Non-Gorenstein terminal Fano varieties of index  $r(X)> n-2$ are classified in \cite{MR1399470}  and have a rational point by Proposition \ref{prop:toric} and \cite[Theorem IV.6.7]{MR1440180}.
 \end{proof}

\section{Terminal Fano threefolds}\label{sec:fano_3fold}

\subsection{Gorenstein case}\label{sec:fano_3fold_gor}
We prove Theorem \ref{thm:3dim_Fano_gor}.

\begin{proof}[Proof of Theorem \ref{thm:3dim_Fano_gor}]
We recall that the index $r(X)$ of a Gorenstein terminal Fano variety $X$ is a positive integer.
Terminal Fano threefolds of index $\geq 2$ are covered by Theorem \ref{thm:largeindex_Fano}. Hence, we need to consider only the cases of index $1$.

$\QQ$-factorial Gorenstein terminal Fano threefolds of Picard rank 1 and geometric Picard rank $>1$ are classified in \cite[Theorem 1.2, \S2, Theorem 6.6]{MR3100918} (cf.~Section \ref{sec:rkpic>1}).

The varieties \cite[(1.2.4), (1.2.7)]{MR3100918} have index 2 by \cite[Remarks (vi) p. 217]{MR1668579}.

The varieties \cite[Case 1.2.1 a), p.~421]{MR3100918} and 
\cite[ (1.2.2), (1.2.8)]{MR3100918} are $k$-models of complete intersections of divisors in products of projective spaces that have a $k$-rational point by \cite[Theorem 1.1]{forms_ci} and Proposition \ref{prop:i_prod_wproj} (cf.~\cite[Remark 2.1]{MR3100918}).

The variety \cite[Case 1.2.1 b), p.~421]{MR3100918} is a  double cover $\pi:X\to V$ of a variety $V$ that belongs to \cite[(1.2.4)]{MR3100918} with branch locus a member of $|-K_V|$. 
The morphism $\pi:X_{\overline k}\to \PP^2_{\overline k}\times\PP^2_{\overline k}$ is defined over $k$ (see \cite[proof of Theorem 6.6 (ii)]{MR3100918}) and the $\gal(\overline k/k)$-action on $V_{\overline k}$ is induced by the one on $\PP^2_{\overline k}\times\PP^2_{\overline k}$. 
We observe that $V$ is smooth by \cite[Theorem 6.6 (i)]{MR3100918}, and $\pic(V_{\overline k})\cong\pic(\PP^2_{\overline k}\times\PP^2_{\overline k})$ by  \cite[Corollary IV.3.3]{MR0282977}. Then $\rk\pic(\PP^2_{\overline k}\times\PP^2_{\overline k})^{\gal(\overline k/k)}=\rk\pic(V)=1$.
The variety $V_{\overline k}$ is a hypersurface of bidegree $(1,1)$ in $\PP^2_{\overline k}\times\PP^2_{\overline k}$. Moreover, $-K_V=(2,2)$ in $\pic(V_{\overline k})$. Hence, the branch locus of $\pi$ is defined by the restriction to $V$ of a hypersurface of degree $(2,2)$ in $\PP^2_{\overline k}\times\PP^2_{\overline k}$. Then $X(k)\neq\emptyset$ by Proposition \ref{prop:cyclic}.

The varieties \cite[(1.2.3)]{MR3100918} have a $k$-rational point by Propositions \ref{prop:(1.2.3)} and \ref{prop:(1.2.3)sing}. 

The variety \cite[(1.2.5)]{MR3100918} is a $k$-model $X$ of a double cover of $\PP^1_{\overline k}\times\PP^1_{\overline k}\times\PP^1_{\overline k}$ with branch locus a member of $|-K_{\PP^1_{\overline k}\times\PP^1_{\overline k}\times\PP^1_{\overline k}}|$ (which is a divisor of tridegree $(2,2,2)$). 
By \cite[Lemma 4.4]{MR3100918}, the composition with the first projection $f_1:X_{\overline k}\to\PP^1_{\overline k}$ is a del Pezzo bundle with general fiber $F_1$, let $\{F_1,F_2,F_3\}$ be the $\gal(\overline k/k)$-orbit of $F_1$ in $\pic(X_{\overline k})$. Then there are other two del Pezzo bundles $f_1,f_2$ conjugate to $f_1$. Then product $f=f_1\times f_2\times f_3:X_{\overline k}\to \PP^1_{\overline k}\times \PP^1_{\overline k}\times \PP^1_{\overline k}$ is a finite map and it is defined over $k$.
Thus $X$ is a double cover of a $k$-model $Y$ of $\PP^1_{\overline k}\times\PP^1_{\overline k}\times\PP^1_{\overline k}$. Since $\rk\pic(X)=1$, also $\rk\pic(Y)=1$ 
 and hence $\pic(Y)$ is the subgroup of $\pic(\PP^1_{\overline k}\times\PP^1_{\overline k}\times\PP^1_{\overline k})$ generated by $(1,1,1)$. Then $X$ has a $k$-point by Proposition \ref{prop:cyclic}.

The variety \cite[(1.2.6)]{MR3100918} is a $k$-model $X$ of an intersection of divisors of tridegrees $(0,1,1)$, $(1,0,1)$, $(1,1,0)$ in $\PP^2_{\overline k}\times\PP^2_{\overline k}\times\PP^2_{\overline k}$. 
Let $\pi_1,\pi_2,\pi_3:\PP^2_{\overline k}\times\PP^2_{\overline k}\times\PP^2_{\overline k}\to \PP^2_{\overline k}\times\PP^2_{\overline k}$ be the three projections.
The $\gal(\overline k/k)$-action on $\PP^2_{\overline k}\times\PP^2_{\overline k}\times\PP^2_{\overline k}$ is by permutation of the factors by Lemma \ref{lem:forms_product_proj_space}, and $\rk\pic(\PP^2_{\overline k}\times\PP^2_{\overline k}\times\PP^2_{\overline k})^{\gal(\overline k/k)}=1$. So, for every $i,j\in\{1,2,3\}$, there exists $g_{i,j}\in\gal(\overline k/k)$ such that $\pi_j=\pi_i\circ g_{i,j}$. 
Then $\pi_i(X)=\pi_i\circ g_{i,j}(X)=\pi_j(X)$ for all $i,j\in\{1,2,3\}$.
By \cite[Case 1.2.6, p.~422]{MR3100918}, $\pi_1(X_{\overline k})\subseteq\PP^2_{\overline k}\times\PP^2_{\overline k}$ is a hypersurface of degree $(1,1)$. So 
 $X_{\overline k}=\pi_1^{-1}(\pi_1(X))\cap\pi_2^{-1}(\pi_1(X))\cap\pi_3^{-1}(\pi_1(X))$.
Since $\pi_j^{-1}(\pi_1(X))=g_{i,j}^{-1}(\pi_i^{-1}(\pi_1(X))$ for every $i,j\in\{1,2,3\}$, $X_{\overline k}$ is a complete intersections of hypersurfaces $H_1,H_2,H_3$  in $\PP^2_{\overline k}\times\PP^2_{\overline k}\times\PP^2_{\overline k}$
of degrees $(0,1,1)$, $(1,0,1)$, $(1,1,0)$ respectively, such that $\{H_1,H_2,H_3\}$ is a $\gal(\overline k/k)$-invariant set under the action of $\gal(\overline k/k)$ over $X_{\overline k}$. 
Hence, $X(k)\neq\emptyset$ by Proposition \ref{prop:i_prod_wproj}.

Indecomposable Gorenstein terminal Fano threefolds of index 1 and geometric Picard rank 1 are classified in \cite[Theorems 1.10,  6.5]{MR1944132} (cf.~Sections \ref{sec:smooth_mukai} and \ref{sec:sing_gor_rkpic_1}).
 Let $X$ be an indecomposable Gorenstein terminal Fano threefold of index 1 and geometric Picard rank 1. We use the notation introduced in Section \ref{sec:smooth_mukai}.

If $g=2$, then 
$X\subseteq\PP(1,1,1,1,3)$ is a hypersurface of degree 6. Thus $X(k)\neq\emptyset$ by  \cite[Theorem IV.6.7]{MR1440180}.

If $g=3$ and $X$ is a quartic hypersurface in $\PP^4_k$, then $X$ has a $k$-rational point by definition of $C_1$ field. If $g=3$ and  $X$ is a double cover of a quadric hypersurface $Q$ in $\PP^4_k$, then it is ramified along the restriction to $Q$ of a hypersurface of degree 4 (see \cite[(4.3.2)]{MR2112596}).
Thus  $X\subseteq\PP(1,1,1,1,1,2)$ is a complete intersection of two hypersurfaces of degrees 2 and 4, and $X(k)\neq\emptyset$ by \cite[Theorem IV.6.7]{MR1440180}.

If $g\in\{4,5\}$ the variety $X$ has a $k$-rational point by  \cite[Theorem 4]{MR0046388}.
 
If $g=12$ and $X$ is singular, then its singular locus consists of a $k$-point by \cite[Theorem 1.3]{MR3535377}.  If $g=12$ and $X$ is smooth, then the Fano variety of conics of $X$ is isomorphic to $\PP^2_k$ by \cite[Proposition B.4.1]{MR3776469} and 
Remark \ref{rem:brauer}. 
Hence, $X$ contains smooth conics defined over $k$,  which have $k$-points again by  Remark \ref{rem:brauer}.
\end{proof}

As recalled in Section \ref{sec:sing_gor_rkpic_1}, 
decomposable terminal Gorenstein Fano threefolds of index 1 and geometric Picard rank 1 are not completely classified. However, the cases with $g\in\{2,3,4,5\}$ 
 are completely classified in \cite[Theorem 6.5]{MR1944132}. Moreover, if $g=12$ the only decomposable cases we are interested in are classified in \cite[Theorem 1.3]{MR3535377}. These cases are covered by Theorem \ref{thm:3dim_Fano_gor}.

\subsection{Non-Gorenstein case}\label{subsection:non-gorenstein}
We start by considering 
non-Gorenstein terminal Fano threefolds of index 1 with only cyclic quotient singularities, which are classified in \cite{MR1317287}.

\begin{theorem}\label{thm:non_gor_cyclic}
Let $X$ be a non-Gorenstein terminal Fano threefold of index 1 with only cyclic quotient singularities over a  $C_1$ field $k$ of characteristic 0 that admits normic forms of arbitrary degree. Then $X(k)\neq\emptyset$.
\end{theorem}
\begin{proof}
Let $\overline k$ be an algebraic closure of $k$.
 By Lemma \ref{lem:index} 
the double covering in \cite[Theorem 1.1]{MR1317287} is 
defined over $k$, we denote it by $Y\to X$. Moreover, $Y$ is a smooth Fano threefold, and there are $14$ possible cases for $Y$.

In the cases \cite[Theorem 1.1, No. 8, 12, 14]{MR1317287} the variety $Y$ has index 2, then $Y(k)\neq\emptyset$ by Theorem \ref{thm:largeindex_Fano}.
In the remaining cases the variety $Y$ has index $1$. If $Y$ has Picard rank $1$ over $k$, then $Y(k)\neq\emptyset$ by Theorem \ref{thm:3dim_Fano_gor}. 
Therefore we can assume that $Y$ has Picard rank $\geq2$ over $k$. 

In the case \cite[Theorem 1.1, No. 3, 7]{MR1317287} (see also \cite[\S12.3, No. 3, 10]{MR1668579}), the variety $Y$ is a blow-up of a smooth Fano threefold of index 2 that has a $k$-point by Theorem \ref{thm:largeindex_Fano}. Then $Y(k)\neq\emptyset$ by \cite{MR0095851}. 

In the cases \cite[Theorem 1.1, No. 4, 10, 13]{MR1317287}, by studying the automorphisms of $Y_{\overline k}$ as in  \cite[\S\S 6.5.1, 6.6.1, 6.6.2]{MR1268578} we conclude that $Y$ is isomorphic to $C\times_k S_d$, where $C$ is a smooth conic and $S_d$ is a smooth del Pezzo surface of degree $d\in\{2, 4, 6\}$. Since $C(k)\neq\emptyset$ by Remark \ref{rem:brauer} and $S_d(k)\neq\emptyset$ by \cite[Proposition 2]{MR934267}, we conclude that  $Y(k)\neq\emptyset$.

In the case \cite[Theorem 1.1, No. 5]{MR1317287} (see also \cite[\S12.4, No. 1]{MR1668579}),
the variety $Y$ is a double covering of a $k$-form $Z$ of $\PP^1_k\times\PP^1_k\times\PP^1_k$ ramified along a divisor of tridegree $(2,2,2)$. 
Since $\rk\pic(Y_{\overline k})=3$ we deduce that $\rk\pic(Z)=\rk\pic(Y)\geq 2$. By Lemma \ref{lem:forms_product_proj_space} this can only happen if $Z$ is a product of a smooth conic $C$ with a $k$-form of $\PP^1_k\times\PP^1_k$. 
By \cite[Lemma 4.4]{MR3100918}, the composition with the projection $Z\to C$ is a del Pezzo bundle structure on $Y$. By Remark \ref{rem:brauer} $C(k)$ is nonempty, and hence Zariski dense in $C$. Thus $Y$ contains a smooth del Pezzo surface defined over $k$ that has a $k$-point by \cite[Proposition 2]{MR934267}.

In the case \cite[Theorem 1.1, No. 6]{MR1317287},
 since $\rk\pic(Y)=2$, the two extremal contractions from \cite[Case ($C_1$,$C_1$), p.140]{MR1668579} are defined over $k$. Hence 
 there is a dominant morphism $Y\to \PP^2_k$ whose general fiber is a smooth conic, and hence contains $k$-points by Remark \ref{rem:brauer}.

In the case \cite[Theorem 1.1, No. 9]{MR1317287}, $Y$ is a $k$-model of a complete intersection of three divisors of tridegree $(1,1)$ in $\PP^3_{\overline k}\times\PP^3_{\overline k}$, hence $Y$ is a complete intersection of three divisors of bidegree $(1,1)$ in a $k$-form of $\PP^3_{k}\times\PP^3_{k}$ by \cite[Theorem 1.1]{forms_ci} and $Y(k)\neq\emptyset$ by Proposition \ref{prop:i_prod_wproj}.

In the case \cite[Theorem 1.1, No. 11]{MR1317287}, $Y$ is a $k$-model of a divisor of multidegree $(1,1,1,1)$ in $\PP^1_{\overline k}\times\PP^1_{\overline k}\times\PP^1_{\overline k}\times \PP^1_{\overline k}$, hence $Y(k)\neq\emptyset$ by Proposition \ref{prop:i_prod_wproj}.
\end{proof}

The Graded Ring Database \cite{GRD} contains the list of possible baskets \cite[\S2]{MR2374989} of non-Gorenstein singularities for geometrically $\QQ$-factorial non-Gorenstein terminal  Fano threefolds of geometric Picard rank 1 over an algebraically closed field of characteristic 0. Such varieties $X$ are studied in \cite{MR2374989} using the Fano index $f(X)$ defined in the introduction.
The Fano index $f(X)$ divides $q(X)$ and they coincide if there is no torsion in the Weil divisor class group (see \cite[Corollary 2.3]{MR2374388} and  \cite[Lemma 3.2]{MR2745685}). In particular, $f(X)\in\{1,\dots,9,11,13,17,19\}$ by \cite{MR2067795}.

\begin{remark}\label{rem:basket}
The basket of singularities of a terminal Fano variety  was introduced in \cite[(8.2), (10.2)]{MR927963}. It  is a collection of quotient singularity germs. To each non-Gorenstein singular point of the variety there is an associated collection of quotient singularity germs.
The disjoint union of such collections forms the basket of the variety. 
The collection of quotient singularity germs associated to a singular point is invariant under the automorphisms of the variety. Therefore, if a given quotient singularity germ appears only once in a basket (we say it has multiplicity one in the basket), then the corresponding singular point is  a fixed point for all the automorphisms of the variety.
\end{remark}

\begin{theorem}\label{thm:non-gor_large}
Let $k$ be a field of characteristic $0$ and $\overline k$ an algebraic closure.
Let $X$ be a 
 non-Gorenstein terminal Fano threefold over $k$ such that $X_{\overline k}$ is $\QQ$-factorial and $\rk\pic(X_{\overline k})=1$.
\begin{enumerate}
\item If  $f(X_{\overline k})\geq 6$, then $X(k)\neq\emptyset$.
\item If $f(X_{\overline k})=5$ and $k$ is $C_1$ and admits normic forms of arbitrary degree, then $X(k)\neq\emptyset$.
\end{enumerate}
\end{theorem}
\begin{proof}
By inspection in the Graded Ring Database \cite[Fano 3-folds]{GRD},
if $f(X_{\overline k})\geq 6$ each possible basket of singularities for $X$ contains a quotient singularity germ that appears with multiplicity one. Hence, the corresponding singular point on $X(\overline k)$ is invariant under the action of $\gal(\overline k/k)$ by Remark \ref{rem:basket}.
If $f(X_{\overline k})=5$ the same argument works for all possible baskets of singularities except one \cite[Fano 3-folds, ID 41439]{GRD}, which is realized by a hypersurface of degree $6$ in  $\PP(1, 2, 2, 3, 3)$, and hence has a $k$-point by \cite[Theorem IV.6.7]{MR1440180}.
\end{proof}

The Graded Ring Database \cite{GRD} contains $1847$ cases with Fano index $f\in\{2,3,4\}$ and $52646$ cases with Fano index $f=1$.
The arguments used in Theorem \ref{thm:non-gor_large} work for all except $9$ cases if $f=4$, for all except $13$ cases if $f=3$, for all except $109$ cases if $f=2$, and for at least $1581$ cases if $f=1$ (the basket has cardinality $1$ in $1415$ cases,  the variety is a Fano complete intersection of codimension at most $2$ in a weighted projective space in $166$ further cases).

\bibliographystyle{alpha}
\bibliography{ratconn_preprints,../../../../../bibliography}

\end{document}